\documentclass{amsart}
\usepackage{amsmath,amssymb}
\usepackage{bbold}
\usepackage{psfrag}
\usepackage{graphicx}
\usepackage{subfigure}
\usepackage{leftidx}
\usepackage[all]{xy}
\usepackage[latin1]{inputenc}
\usepackage{enumerate}

\newcommand{\lMod}[1]{\leftidx{_{#1}}{\mathrm{Mod}}{}}
\newcommand{\FO}{{\mathbb H}}

\newcommand{\AT}{\mathbb{A}}

\newcommand{\YD}[1]{\leftidx{_{#1}^{#1}}{\mathcal{YD}}{}}

\DeclareMathOperator{\cdiv}{{\div \hspace*{-.2em}\raisebox{.15\height}{\scalebox{1}[.573]{$\mid$}}}}

\newtheorem{thm}{Theorem}[section]
\newtheorem{cor}[thm]{Corollary}

\newtheorem{prop}[thm]{Proposition}

\theoremstyle{definition}

\newtheorem{rem}[thm]{Remark}
\newtheorem{exa}[thm]{Example}

\newcommand{\HopfAlg}{\mathrm{HopfAlg}}
\newcommand{\HopfMon}{\mathrm{HopfMon}}

\newcommand{\HAr}{\mathbb{H}_r}
\newcommand{\HAl}{\mathbb{H}_l}
\newcommand{\HAri}{\mathbb{H}_r^{-1}}
\newcommand{\HAli}{\mathbb{H}_l^{-1}}

\newcommand{\co}{\colon}
\newcommand{\id}{\mathrm{id}}
\newcommand{\kk}{\Bbbk}
\newcommand{\kt}{$\Bbbk$\nobreakdash-\hspace{0pt}}

\newcommand{\opp}{\mathrm{op}}

\newcommand{\cc}{\mathcal{C}}
\newcommand{\bb}{\mathcal{B}}

\newcommand{\dd}{\mathcal{D}}

\newcommand{\uu}{\mathcal{U}}

\newcommand{\zz}{\mathcal{Z}}

\newcommand{\rmod}[2]{\leftidx{}{#1}{_{#2}}}
\newcommand{\lmod}[2]{\leftidx{_{#2}}{#1}{}}

\newcommand{\Rt}{$\mathrm{R}$\nobreakdash-\hspace{0pt}}
\newcommand{\ti}{\mbox{-}\,}
\newcommand{\un}{\mathbb{1}}
\newcommand{\cp}{\rtimes}

\newcommand{\Ob}{\mathrm{Ob}}

\newcommand{\Vect}{{\mathrm{Vect}}}
\newcommand{\vect}{\mathrm{vect}}

\newcommand{\lev}{\mathrm{ev}}
\newcommand{\rev}{\widetilde{\mathrm{ev}}}
\newcommand{\lcoev}{\mathrm{coev}}
\newcommand{\rcoev}{\widetilde{\mathrm{coev}}}

\newcommand{\ldual}[1]{\leftidx{^\vee}{\!#1}{}}
\newcommand{\rdual}[1]{{#1}^\vee}

\newcommand{\adjunct}[2]{\!\!\raisebox{.6ex}{\xymatrix{\ar@/^.4pc/[r]^{#1}  &  \ar@/^.4pc/[l]^{#2}}}\!\!}
\newcommand{\rsdraw}[3]{\raisebox{-#1\height}{\scalebox{#2}{\includegraphics{#3.eps}}}}
\providecommand{\bysame}{\leavevmode\hbox to3em{\hrulefill}\thinspace}

\newcommand{\mo}[2]{{#2}^{#1}}

\usepackage[dvipsnames]{xcolor}

\begin{document}
\title{The doubles of a braided Hopf algebra}
\author[A. Brugui\`eres]{Alain Brugui\`eres}
\author[A. Virelizier]{Alexis Virelizier}
\email{bruguier@math.univ-montp2.fr \and virelizi@math.univ-montp2.fr}

\date{\today}

\begin{abstract} Let $A$ be a Hopf algebra in a braided rigid category $\bb$. In the case $\bb$ admits a coend $C$, which is a Hopf algebra in $\bb$, we defined in 2008 the double $D(A)=A \otimes \ldual{A} \otimes C$ of $A$, which is a quasitriangular Hopf algebra in $\bb$ whose category of modules is isomorphic to the center of the category of $A$\ti modules as a braided category.
Here, quasitriangular means endowed with an \Rt matrix (our notion of \Rt matrix for a Hopf algebra in $\bb$ involves the coend $C$ of $\bb$). In general, i.e.\@ when $\bb$ does not necessarily admit a coend, we construct a quasitriangular Hopf monad $d_A$ on the center $\zz(\bb)$ of $\bb$ whose category of modules is isomorphic to the center of the category of $A$\ti modules as a braided category. As an endofunctor of $\zz(\bb)$, $d_A$ it is given by $X \mapsto X \otimes A \otimes \ldual{A}$. We prove that the Hopf monad $d_A$ may not be representable by a Hopf algebra. If $\bb$ has a coend $C$, then $D(A)$ is the cross product of the Hopf monad $d_A$ by~$C$. Equivalently, $d_A$ is the cross quotient of $D(A)$ by $C$.
\end{abstract}
\maketitle

\setcounter{tocdepth}{1} \tableofcontents

\section*{Introduction}

The notion of Hopf algebra extends naturally to the context of braided categories. Hopf algebras in braided categories, also called braided Hopf algebras, have been studied by many authors and it has been shown that several aspects of the theory of Hopf algebras can be extended to this setting. This paper is devoted to the study of the double of a braided Hopf algebra.

Let $A$ be a Hopf algebra in a braided rigid category $\bb$. Assume first that $\bb$ admits a coend $C$. Then $C$ is a Hopf algebra in $\bb$ such that the center $\zz(\bb)$ of $\bb$ is isomorphic to the category of right $C$-modules in $\bb$:
$$
\zz(\bb) \simeq \bb_{C}.
$$
In \cite{BV3}, we defined the notion of an \Rt matrix for $A$, which is a morphism $$\mathfrak{r}\co C \otimes C \to A \otimes A$$ satisfying certain axioms generalizing those of Drinfeld. These \Rt matrixes are in one-to-one correspondence with braidings on the category of $A$\ti modules.
This notion of \Rt matrix is different from that previously introduced by Majid in \cite{Maj3}, which did not involve the coend, and did not have this property. We also defined the double $D(A)$ of $A$. As an object, $D(A)= A \otimes \ldual{A} \otimes C$, where $\ldual{A}$ is  the left dual of $A$. The double $D(A)$ is a Hopf algebra in $\bb$ which is quasitriangular (i.e., endowed with an \Rt matrix) and whose category of modules is isomorphic to the categorical center of the category of right $A$\ti modules in $\bb$ as a braided category:
$$
\zz(\bb_A) \simeq \bb_{D(A)}.
$$

What happens when $\bb$ does not necessarily admit a coend?
There is a canonical strict monoidal functor
$
\uu \co \zz(\bb_A) \to \zz(\bb)
$
from the center of the category of right $A$\ti modules to the center $\zz(\bb)$ of $\bb$. We prove that this functor is monadic. As an endofunctor of $\zz(\bb)$, the associated monad $d_A$ is given by $X \mapsto X \otimes A \otimes \ldual{A}$. The monad $d_A$ is a quasitriangular Hopf monad (described explicitly) whose category of modules is isomorphic to the center of the category of $A$\ti modules as a braided category:
$$
\zz(\bb_A) \simeq \zz(\bb)^{d_A}.
$$
Recall that Hopf monads, which were introduced in \cite{BV2} and further studied in \cite{BLV}, are algebraic objects which generalize Hopf algebras in braided categories to the setting of monoidal categories. We call $d_A$ the \emph{central double of $A$}.

The central double $d_A$ of $A$ is not in general representable by a Hopf algebra, and so $\zz(\bb_A)$ cannot be described as a category of modules over some Hopf algebra in $\zz(\bb)$. To prove this,
we show that the center of the category of $A$\ti modules can be described as a category of Yetter-Drinfeld modules for $A$, viewed as a Hopf algebra in $\zz(\bb)$. (Actually for this result $\bb$ needs not be braided, if we assume that $A$ is a Hopf algebra in the center of $\bb$).

When $\bb$ has a coend $C$, what is the relationship between the double $D(A)$ and the central double $d_A$ of $A$?
Recall that the notions of cross-product and cross-quotient of Hopf monads were introduced in \cite{BV3} and \cite{BLV} respectively. Viewing the Hopf algebras $D(A)$ and $C$ as Hopf monads, we prove that $D(A)$ is the cross product of $d_A$ by $C$ and so that $d_A$ is the cross quotient of $D(A)$ by $C$:
$$
D(A)=d_A \rtimes C \quad \text{and} \quad d_A=D(A)\cdiv C.
$$
Note that this is an illustration of the fact the cross-quotient of two Hopf monads representable by Hopf algebras is not always representable by a Hopf algebra.

This paper is organized as follows. In Section~\ref{sect-prelims}, we review several facts about monoidal categories, braided categories, the center construction, and coends. In Section~\ref{sect-HAs}, we recall the definition of \Rt matrices and the construction of the double of a braided Hopf algebra from \cite{BV3}. Using Yetter-Drinfeld modules, we prove that the center of category of modules may not be a category of modules over a Hopf algebra in the center.  In Section~\ref{sect-Hopf-Monads-resume}, we recall the definition and some properties of Hopf monads. Section~\ref{sect-central-double} is devoted to the construction of the central double of a braided Hopf algebra. Finally, in Section~\ref{sect-cp-cq}, we study the relationship between the double and the central double of a braided Hopf algebra via cross-product and cross-quotient of Hopf monads.

\section{Preliminaries on categories}\label{sect-prelims}


\subsection{Monoidal categories and monoidal functors}\label{sect-monofunctor}
Given an object $X$ of a monoidal category $\cc$, we denote by $X \otimes ?$ the endofunctor of $\cc$ defined on objects by $Y \mapsto X \otimes Y$ and on morphisms by $f \mapsto \id_X \otimes f$. Similarly one defines the endofunctor $? \otimes X$ of $\cc$.

Let $(\cc,\otimes,\un)$ and $(\dd, \otimes, \un)$ be two monoidal categories.
A \emph{monoidal functor} from $\cc$ to $\dd$ is a triple
$(F,F_2,F_0)$, where $F\co \cc \to \dd$ is a functor, $F_2\co
F\otimes F \to F \otimes$ is a natural transformation, and $F_0\co\un
\to F(\un)$ is a morphism in~$\dd$, such that:
\begin{align*}
& F_2(X,Y \otimes Z) (\id_{F(X)} \otimes F_2(Y,Z))= F_2(X \otimes Y, Z)(F_2(X,Y) \otimes \id_{F(Z)}) ;\\
& F_2(X,\un)(\id_{F(X)} \otimes F_0)=\id_{F(X)}=F_2(\un,X)(F_0
\otimes \id_{F(X)}) ;
\end{align*}
for all objects $X,Y,Z$ of $\cc$.

A monoidal functor $(F,F_2,F_0)$ is said to be \emph{strong} (resp.\@ \emph{strict}) if $F_2$ and $F_0$ are
isomorphisms (resp.\@ identities).

A natural transformation $\varphi\co F \to G$ between monoidal functors is \emph{monoidal}
if it satisfies:
\begin{equation*}
\varphi_{X \otimes Y} F_2(X,Y)= G_2(X,Y) (\varphi_X \otimes
\varphi_Y) \quad \text{and} \quad G_0=\varphi_\un F_0.
\end{equation*}


\subsection{Graphical conventions} We represent morphisms in a category by diagrams to be read from bottom to top.   Thus we draw the identity $\id_X$ of an object $X$, a morphism $f\co X \to Y$, and its composition with a morphism $g\co Y \to Z$  as follows:
\begin{center}
\psfrag{X}[Bc][Bc]{\scalebox{.8}{$X$}} \psfrag{Y}[Bc][Bc]{\scalebox{.8}{$Y$}} \psfrag{h}[Bc][Bc]{\scalebox{.8}{$f$}} \psfrag{g}[Bc][Bc]{\scalebox{.8}{$g$}}
\psfrag{Z}[Bc][Bc]{\scalebox{.8}{$Z$}} $\id_X=$ \rsdraw{.45}{.9}{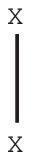}\,,\quad $f=$ \rsdraw{.45}{.9}{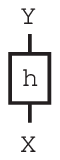} ,\quad \text{and} \quad $gf=$ \rsdraw{.45}{.9}{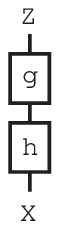}\,.
\end{center}
In a monoidal category, we represent the monoidal product of two morphisms $f\co X \to Y$ and $g \co U \to V$ by juxtaposition:
\begin{center}
\psfrag{X}[Bc][Bc]{\scalebox{.8}{$X$}} \psfrag{h}[Bc][Bc]{\scalebox{.8}{$f$}}
\psfrag{Y}[Bc][Bc]{\scalebox{.8}{$Y$}}  $f\otimes g=$ \rsdraw{.45}{.9}{morphism} \psfrag{X}[Bc][Bc]{\scalebox{.8}{$U$}} \psfrag{g}[Bc][Bc]{\scalebox{.8}{$g$}}
\psfrag{Y}[Bc][Bc]{\scalebox{.8}{$V$}} \rsdraw{.45}{.9}{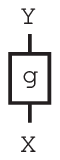}\,.
\end{center}

\subsection{Duals and rigid categories}
Let $\cc$ be a monoidal category. Recall that a \emph{left dual} of an object $X$ of $\cc$ is an object $\ldual{X}$ of $\cc$ endowed with morphisms $\lev_X\co \ldual{X} \otimes X \to \un$ (the \emph{left evaluation}) and  $\lcoev_X\co \un \to X \otimes \ldual{X}$
(the \emph{left coevaluation}) such that
\begin{equation*}
(\lev_X \otimes \id_{\ldual{X}})(\id_{\ldual{X}} \otimes \lcoev_X)=\id_{\ldual{X}} \quad \text{and} \quad (\id_X \otimes \lev_X)(\lcoev_X \otimes \id_X)=\id_X.
\end{equation*}
Likewise, a \emph{right dual} of an object $X$ of $\cc$ is an object $\rdual{X}$ of $\cc$ endowed with morphisms $\rev_X\co X \otimes \rdual{X} \to \un$ (the \emph{right evaluation}) and  $\rcoev_X\co \un \to  \rdual{X} \otimes X$
(the \emph{right coevaluation}) such that
\begin{equation*}
(\rev_X \otimes \id_X)(\id_X \otimes \rcoev_X)=\id_X \quad \text{and} \quad (\id_{\rdual{X}} \otimes \rev_X)(\rcoev_X \otimes \id_{\rdual{X}})=\id_{\rdual{X}}.
\end{equation*}
Left and right duals, if they exist, are unique up to unique isomorphisms preserving the evaluation and coevaluation morphisms.

A \emph{rigid category} is a monoidal category where every objects admits both a left dual and a right dual.
The duality morphisms of a rigid category are depicted as:
\begin{center}
\psfrag{C}[Bc][Bc]{\scalebox{.8}{$X$}} \psfrag{A}[Bc][Bc]{\scalebox{.8}{$\ldual{X}$}} $\lev_X=$ \rsdraw{.45}{.9}{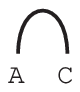}\,,\quad
\psfrag{A}[Bc][Bc]{\scalebox{.8}{$X$}} \psfrag{C}[Bc][Bc]{\scalebox{.8}{$\ldual{X}$}} $\lcoev_X=$ \rsdraw{.45}{.9}{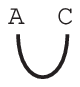}\,,\quad
\psfrag{A}[Bc][Bc]{\scalebox{.8}{$X$}} \psfrag{C}[Bc][Bc]{\scalebox{.8}{$\rdual{X}$}} $\rev_X=$ \rsdraw{.45}{.9}{eval}\,,\quad \text{and} \quad
\psfrag{C}[Bc][Bc]{\scalebox{.8}{$X$}} \psfrag{A}[Bl][Bl]{\scalebox{.8}{$\rdual{X}$}} $\rcoev_X=$ \rsdraw{.45}{.9}{coeval}\,.
\end{center}

\subsection{Braided categories}
A \emph{braiding} of a monoidal category $\bb$ is a natural isomorphism
$ \tau=\{\tau_{X,Y} \co  X \otimes Y\to Y \otimes X\}_{X,Y \in
\bb} $ such that
\begin{equation*}
\tau_{X, Y\otimes Z}=(\id_Y \otimes \tau_{X, Z})(\tau_{X, Y} \otimes \id_Z) \quad \text{and} \quad
\tau_{X\otimes Y,Z}=(\tau_{X,Z} \otimes \id_Y)(\id_X \otimes \tau_{Y,Z})
\end{equation*}
for all $X,Y,Z\in\Ob(\bb)$. These conditions imply that $\tau_{X,\un}=\tau_{\un,X}=\id_X$.

A \emph{braided category} is a monoidal category endowed with a braiding.

The braiding $\tau$ of a braided category, and its inverse, are depicted as
\begin{center}
\psfrag{X}[Bc][Bc]{\scalebox{.8}{$X$}} \psfrag{Y}[Bc][Bc]{\scalebox{.8}{$Y$}} $\tau_{X,Y}=\,$\rsdraw{.45}{.9}{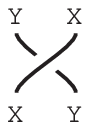} \quad \text{and} \quad $\tau^{-1}_{Y,X}=\,$\rsdraw{.45}{.9}{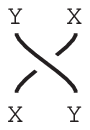}.
\end{center}

If $\bb$ is a braided category with braiding $\tau$, then the \emph{mirror of $\bb$} is the  braided category which coincides with  $\bb$ as a monoidal category and equipped with the braiding  $\overline{\tau}$ defined by
$\overline{\tau}_{X,Y}=\tau^{-1}_{Y,X}$.

\subsection{The center of a  monoidal category}\label{sect-centerusual}
Let $\cc$ be a monoidal category. A \emph{half braiding} of $\cc$ is
a pair $({{A}},\sigma)$, where ${{A}}$ is an object of $\cc$ and
\begin{equation*}
\sigma=\{\sigma_X \co  {{A}} \otimes X\to X \otimes {{A}}\}_{X \in \cc}
\end{equation*}
is a natural isomorphism such that
 \begin{equation}
 \sigma_{X \otimes Y}=(\id_X \otimes
\sigma_Y)(\sigma_X \otimes \id_Y)
\end{equation} for all
$X,Y\in\Ob(\cc)$. This implies that $\sigma_\un=\id_{{A}}$.

The \emph{center of $\cc$} is the braided category $\zz(\cc)$
defined as follows. The objects of~$\zz(\cc)$ are half braidings of
$\cc$. A morphism $({{A}},\sigma)\to ({{A}}',\sigma')$ in $\zz(\cc)$ is a
morphism $f \co {{A}} \to {{A}}'$ in $\cc$ such that $(\id_X \otimes
f)\sigma_X=\sigma'_X(f \otimes \id_X)$ for any object $X$ of $\cc$. The
  unit object of $\zz(\cc)$ is $\un_{\zz(\cc)}=(\un,\{\id_X\}_{X \in
\cc})$ and the monoidal product is
\begin{equation*}
({{A}},\sigma) \otimes ({{B}}, \rho)=\bigl({{A}} \otimes {{B}},(\sigma \otimes \id_{{B}})(\id_{{A}} \otimes \rho) \bigr).
\end{equation*}
The braiding $\tau$ in $\zz(\cc)$ is defined by
$$
\tau_{({{A}},\sigma),({{B}}, \rho)}=\sigma_{{{B}}}\co ({{A}},\sigma) \otimes
({{B}}, \rho) \to ({{B}}, \rho) \otimes ({{A}},\sigma).
$$
If $\cc$ is rigid, then so is $\zz(\cc)$.

The {\it forgetful functor} $\uu\co\zz(\cc)\to \cc$,
$({{A}}, \sigma) \mapsto {{A}}$, is strict monoidal.

If $\bb$ is a braided category, then its braiding $\tau$ defines a fully faithful braided functor $$\left\{\begin{array}{ccl} \bb & \to &\zz(\bb)\\ X & \mapsto &(X,\tau_{X,-})\end{array} \right.$$
which is a monoidal section of the forgetful functor $\zz(\bb) \to \bb$.

\subsection{Coends}\label{sect-coend}
Let $\cc$ and $\dd$ be categories. A \emph{dinatural
transformation} from  a functor $F\co \dd^\opp \times \dd \to \cc$
to an object $A$ of $\cc$  is a family of morphisms in~$\cc$
$$
d=\{d_Y \co F(Y,Y) \to A\}_{Y \in \dd}
$$
such that for every morphism $f\co X \to Y$ in~$\dd$, we have
$$
d_X F(f, \id_X)=d_Y F(\id_Y,f)\colon F(Y,X)\to A.
$$
A \emph{coend} of~$F$
is a pair $(C,\rho)$ consisting  in an object $C$ of $\cc$ and a
dinatural transformation $\rho$ from $F$ to $C$ satisfying the
following universality condition:  for each  dinatural transformation
$d$ from $F$ to an object $c$ of $\cc$, there exists a unique morphism $\tilde{d} : C \to c$ such that
$d_Y = \tilde{d} \rho_Y$ for any $Y$ in $\dd$.
Thus if $F$ has a
coend $(C,\rho)$, then it is unique up to unique isomorphism. One
writes $C= \int^{Y \in \dd}F(Y,Y)$. For more details on coends, see \cite{ML1}.

\section{Hopf algebras in braided categories}\label{sect-HAs}

\subsection{Hopf algebras}
Let $\cc$ be a monoidal category. Recall that an \emph{algebra in $\cc$} is an object $A$ of $\cc$ endowed with morphisms $m\co A \otimes A \to A$ (the product) and $u\co \un \to A$ (the unit) such that
\begin{equation*}
m(m \otimes \id_A)=m(\id_A \otimes m) \quad \text{and} \quad m(\id_A \otimes u)=\id_A=m(u \otimes \id_A).
\end{equation*}
A \emph{coalgebra in $\cc$} is an object $C$ of $\cc$ endowed with morphisms $\Delta\co C \to C \otimes C$ (the coproduct) and $\varepsilon\co C \to \un$ (the counit) such that
\begin{equation*}
(\Delta \otimes \id_C)\Delta=(\id_C \otimes \Delta)\Delta \quad \text{and} \quad (\id_C \otimes \varepsilon)\Delta=\id_C=(\varepsilon \otimes \id_C)\Delta.
\end{equation*}

Let $\bb$ be a braided category, with braiding
$\tau$. A \emph{bialgebra in $\bb$} is an object $A$ of~$\bb$ endowed with an algebra structure (in $\bb$) and a coalgebra structure (in~$\bb$) such that its coproduct $\Delta$ and counit $\varepsilon$ are algebra morphisms (or equivalently, such that it product $m$ and unit $u$ are coalgebra morphisms), that is,
\begin{align*}
\Delta m&=(m \otimes m)(\id_A \otimes \tau_{A,A} \otimes \id_A)(\Delta \otimes \Delta),  & \Delta u&=u \otimes u, \\  \varepsilon m&=\varepsilon \otimes \varepsilon, & \varepsilon u&=\id_\un.
\end{align*}

An \emph{antipode} for a bialgebra $A$ is a morphism $S\co A \to A$ in $\bb$ such that
\begin{equation*}
m(S \otimes \id_A)\Delta=u \varepsilon=m(\id_A \otimes S)\Delta.
\end{equation*}
If it exists, an antipode is unique. A \emph{Hopf algebra in $\bb$} is a bialgebra in $\bb$ which admits an invertible antipode.

Given a Hopf algebra $A$ in a braided category, we depict its product $m$, unit $u$,
coproduct $\Delta$, counit $\varepsilon$, antipode $S$, and $S^{-1}$ as follows:
\begin{center}
\psfrag{A}[Bc][Bc]{\scalebox{.8}{$A$}} $m=$\rsdraw{.45}{.9}{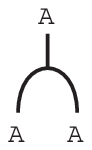}, \quad $u=$\rsdraw{.45}{.9}{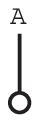}, \quad
$\Delta=$\rsdraw{.45}{.9}{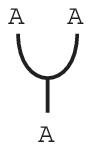}, \quad $\varepsilon=$\rsdraw{.45}{.9}{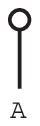}, \quad $S=\,\rsdraw{.45}{.9}{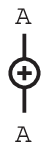}$\,,
\quad $S^{-1}=\,\rsdraw{.45}{.9}{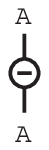}$\,.
\end{center}

\begin{rem}\label{rem-HA-fusion}
Let $A$ be a bialgebra in a braided category $\bb$.
Then $A$ is a Hopf algebra if and only if its left
\emph{left fusion operator}
\begin{center}
$\HAl=(\id_A \otimes m)(\Delta \otimes \id_A) = \psfrag{A}[Bc][Bc]{\scalebox{.8}{$A$}} \rsdraw{.45}{.9}{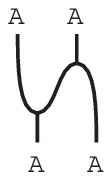}  \co A\otimes A \to A \otimes A
$
\end{center}
and its \emph{right fusion operator}
\begin{center}
$\HAr=(m \otimes \id_A) (\id_A \otimes \tau_{A,A})(\Delta \otimes \id_A) = \psfrag{A}[Bc][Bc]{\scalebox{.8}{$A$}} \psfrag{X}[Bc][Bc]{\scalebox{1}{$\tau_{A,A}$}} \rsdraw{.45}{.9}{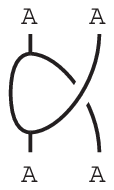}\co A\otimes A \to A \otimes A
$
\end{center}
are invertible. If such is the case,  the inverses of the fusion operators and the antipode $S$ of $A$ and its inverse are related by
$$
\HAli= \psfrag{A}[Bc][Bc]{\scalebox{.8}{$A$}} \rsdraw{.45}{.9}{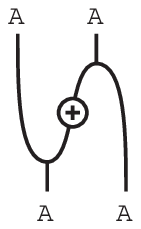},\qquad \HAri= \psfrag{A}[Bc][Bc]{\scalebox{.8}{$A$}} \rsdraw{.45}{.9}{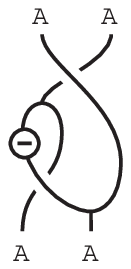}\,, \qquad
 S= \,\psfrag{A}[Bc][Bc]{\scalebox{.8}{$A$}} \psfrag{B}[Bc][Bc]{\scalebox{1}{$\HAli$}} \rsdraw{.45}{.9}{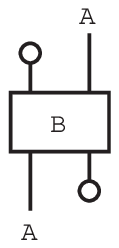}\,,
 \qquad S^{-1}= \,\psfrag{A}[Bc][Bc]{\scalebox{.8}{$A$}} \psfrag{B}[Bc][Bc]{\scalebox{1}{$\HAri$}} \rsdraw{.45}{.9}{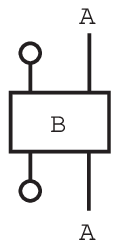}\,.
$$
\end{rem}

\subsection{Modules in categories}\label{sect-modulcatclassicHA}
Let $(A,m,u)$ be an algebra in a monoidal category~$\cc$.
A \emph{left $A$\ti module} (in $\cc$) is a pair~$(M,r)$, where $M$ is an object of $\cc$ and $r\co A \otimes M \to M$ is a morphism in $\cc$, such that
\begin{equation*}
r(m \otimes \id_M)=r(\id_A \otimes r) \quad \text{and} \quad r(u \otimes \id_M)=\id_M.
\end{equation*}
An \emph{$A$\ti linear morphism} between two left $A$\ti modules $(M,r)$ and $(N,s)$ is a morphism $f\co M \to N$ such that $fr=s(\id_A \otimes f)$. Hence the category $\lmod{\cc}{A}$  of left $A$\ti modules. Likewise, one defines the category $\rmod{\cc}{A}$ of right $A$\ti modules.

Let $A$ be a bialgebra in a braided category $\bb$.  Then the category $\lmod{\bb}{A}$ is monoidal, with unit object $(\un,\varepsilon)$ and monoidal product
\begin{equation*}
(M,r) \otimes (N,s)= (r \otimes s)(\id_A \otimes \tau_{A,M} \otimes \id_N)(\Delta \otimes \id_{M \otimes N}),
\end{equation*}
where $\Delta$ and $\epsilon$ are the coproduct and counit of $A$, and $\tau$ is the braiding of $\bb$.
Likewise the category $\rmod{\bb}{A}$ is monoidal, with unit object $(\un,\varepsilon)$ and monoidal product:
\begin{equation*}
(M,r) \otimes (N,s)= (r \otimes s)(\id_M \otimes \tau_{N,A} \otimes \id_A)(\Delta \otimes \id_{M \otimes N}).
\end{equation*}

Assume $\bb$ is rigid. Then $\lmod{\bb}{A}$ is rigid if and only if $\rmod{\bb}{A}$ is rigid, if and only if $A$ is a Hopf algebra.
If $A$ is a Hopf algebra, with antipode~$S$, then the duals of a left $A$-module $(M,r)$ are:
\begin{align*}
&\ldual{(M,r)}=\bigl(\ldual{M}, (\lev_M \otimes \id_{\ldual{M}}) (\id_{\ldual{M}} \otimes r(S \otimes \id_M)\otimes \id_{\ldual{M}})(\tau_{A,\ldual{M}} \otimes \lcoev_M)\bigr),\\
&\rdual{(M,r)}=\bigl(\rdual{M},  (\id_{\rdual{M}} \otimes \rev_M)(\id_{\rdual{M}} \otimes r\tau^{-1}_{A,M} \otimes \id_{\rdual{M}})(\rcoev_M \otimes S^{-1} \otimes \id_{\rdual{M}})\bigr),
\end{align*}
and the duals of a right $A$-module $(M,r)$ are:
\begin{align*}
&\ldual{(M,r)}=\bigl(\ldual{M},  (\lev_M \otimes \id_{\ldual{M}})(\id_{\ldual{M}} \otimes r\tau^{-1}_{M,A} \otimes \id_{\ldual{M}})(\id_{\ldual{M}} \otimes S^{-1} \otimes \lcoev_M) \bigr),\\
&\rdual{(M,r)}=\bigl(\rdual{M}, (\id_{\rdual{M}} \otimes \rev_M) (\id_{\rdual{M}} \otimes r(\id_M \otimes S) \otimes \id_{\rdual{M}})(\rcoev_M \otimes \tau_{\rdual{M},A}) \bigr).
\end{align*}

\begin{rem}\label{rem-braidleftright}
Let $A$ be a Hopf algebra in a braided category $\bb$, with braiding~$\tau$. The functor $F_A\co \lmod{\bb}{A} \to \rmod{\bb}{A}$, defined by $F_A(M,r)=\bigl(M,r \tau_{M,A}(\id_M \otimes S)\bigr)$ and $F_A(f)=f$, gives rise to a strong monoidal isomorphism of categories:
$$F_A=(F_A,\tau, \un) \co (\lmod{\bb}{A})^{\otimes\opp} \to \rmod{\bb}{A}.$$
Therefore braidings on $\lmod{\bb}{A}$ are in bijection with braidings on $\rmod{\bb}{A}$. More precisely, if $c$ is a braiding on  $\rmod{\bb}{A}$, then:
$$c'_{(M,r),(N,s)}=\tau_{M,N}\,c_{F_A(N,s),F_A(M,r)}\,\tau^{-1}_{N,M}$$ is a braiding on $\lmod{\bb}{A}$ (making $F_A$ braided), and the correspondence $c \mapsto c'$ is bijective.
\end{rem}

\subsection{The coend of a braided rigid category}\label{sect-coend-category}
Let $\bb$ be braided rigid category. The coend
\begin{equation*}
C=\int^{Y \in \bb} \leftidx{^\vee}{Y}{} \otimes Y,
\end{equation*}
if it exists, is called the \emph{coend of  $\bb$}.

Assume $\bb$ has a coend $C$ and denote by
$i_Y\co \leftidx{^\vee}{Y}{}  \otimes Y \to C$ the corresponding universal dinatural transformation. The
\emph{universal coaction} of $C$ on the objects of~$\bb$ is the natural transformation $\delta$ defined by
\begin{equation*}
\delta_Y=(\id_Y \otimes i_Y)(\lcoev_Y \otimes \id_Y)\co Y \to Y \otimes C, \quad \text{depicted as} \quad \psfrag{C}[Bc][Bc]{\scalebox{.8}{$C$}} \psfrag{Y}[Bc][Bc]{\scalebox{.8}{$Y$}}\delta_Y=\rsdraw{.45}{.95}{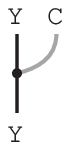}.
\end{equation*}
As shown by Majid~\cite{Maj2}, $C$ is a Hopf algebra in $\bb$. Its coproduct $\Delta$, product $m$,  counit $\varepsilon$, unit $u$, and antipode $S$ with inverse $S^{-1}$ are characterized by the following equalities, where $X,Y\in\bb$:
\begin{gather*}
\psfrag{Y}[Bc][Bc]{\scalebox{.8}{$Y$}}
\psfrag{C}[Bc][Bc]{\scalebox{.8}{$C$}}
\psfrag{D}[cc][cc]{\scalebox{.9}{$\Delta$}}
\rsdraw{.45}{.9}{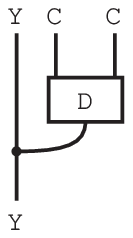} \, = \;\, \rsdraw{.45}{.9}{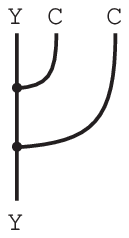} , \quad
\psfrag{Y}[Bc][Bc]{\scalebox{.8}{$Y$}}
\psfrag{C}[Bc][Bc]{\scalebox{.8}{$C$}}
\psfrag{D}[cc][cc]{\scalebox{1}{$\varepsilon$}}
\quad\quad\rsdraw{.45}{.9}{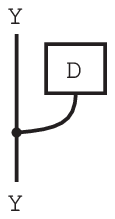} \; = \; \psfrag{Y}[Bc][Bc]{\scalebox{.8}{$Y$}}\rsdraw{.45}{.9}{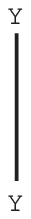} \; , \qquad \;\;
\psfrag{Y}[Bc][Bc]{\scalebox{.8}{$Y$}}
\psfrag{X}[Bc][Bc]{\scalebox{.8}{$X$}}
\psfrag{Z}[Bc][Bc]{\scalebox{.8}{$X\otimes Y$}}
\psfrag{C}[Bc][Bc]{\scalebox{.8}{$C$}}
\psfrag{m}[cc][cc]{\scalebox{.9}{$m$}}
\rsdraw{.45}{.9}{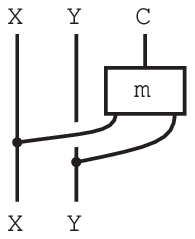} \, = \;\, \rsdraw{.45}{.9}{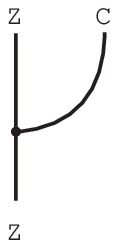} \,,\\[.4em]
 u=\delta_\un, \qquad\;
\psfrag{a}[Bc][Bc]{\scalebox{.9}{$\lev_{Y}$}}
\psfrag{u}[Bc][Bc]{\scalebox{.9}{$\lcoev_{Y}$}}
\psfrag{Y}[Bl][Bl]{\scalebox{.8}{$Y$}}
\psfrag{C}[Bl][Bl]{\scalebox{.8}{$C$}}
\psfrag{D}[cc][cc]{\scalebox{.9}{$S$}} \rsdraw{.45}{.9}{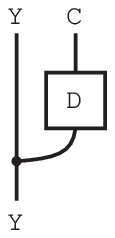} \; = \,
\rsdraw{.45}{.9}{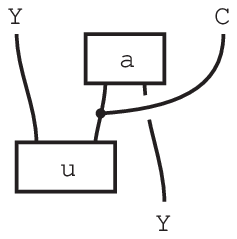}\, ,
\qquad
\psfrag{a}[Bc][Bc]{\scalebox{.9}{$\rev_{Y}$}}
\psfrag{u}[Bc][Bc]{\scalebox{.9}{$\rcoev_{Y}$}}
\psfrag{Y}[Bl][Bl]{\scalebox{.8}{$Y$}}
\psfrag{C}[Bl][Bl]{\scalebox{.8}{$C$}}
\psfrag{D}[cc][cc]{\scalebox{.9}{$S^{-1}$}} \rsdraw{.45}{.9}{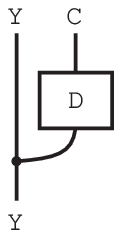} \; = \,
\rsdraw{.45}{.9}{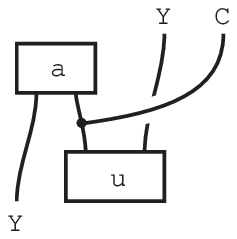}\,.
\end{gather*}
Furthermore, the morphism $\omega\co C \otimes C \to \un$ defined by
\begin{equation*}
\psfrag{Y}[Bl][Bl]{\scalebox{.8}{$Y$}}\psfrag{X}[Bc][Bc]{\scalebox{.8}{$X$}}
\psfrag{C}[Bl][Bl]{\scalebox{.8}{$C$}}
\psfrag{w}[cc][cc]{\scalebox{.9}{$\omega$}}
\rsdraw{.45}{.9}{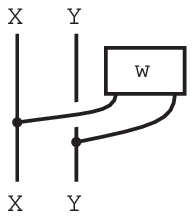} \, = \;\, \rsdraw{.45}{.9}{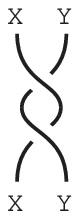}
\end{equation*}
is a Hopf pairing for $C$, called the \emph{canonical pairing}. This means that
\begin{align*}
&\omega(m \otimes \id_C)=\omega (\id_C \otimes \omega \otimes \id_C)(\id_{C^{\otimes 2}} \otimes \Delta), && \omega(u
\otimes
\id_C)=\varepsilon,\\
&\omega(\id_C \otimes m)=\omega (\id_C \otimes \omega \otimes \id_C)(\Delta \otimes \id_{C^{\otimes 2}}), &&
\omega(\id_C \otimes u)=\varepsilon.
\end{align*}
These axioms imply: $\omega(S \otimes \id_C)=\omega(\id_C \otimes S)$. Moreover the canonical pairing $\omega$ satisfies the self-duality condition:    $\omega \tau_{C,C} (S \otimes
S)=\omega$.

In the following, the structural morphisms of $C$ are drawn in grey and the
Hopf pairing $w\co C \otimes C \to \un$ is depicted as:
\begin{center}
\psfrag{C}[Bc][Bc]{\scalebox{.8}{$C$}} $\omega=$\rsdraw{.45}{.9}{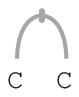}.
\end{center}
\begin{rem}
The category $\bb$ is symmetric if and only if $\omega=\epsilon \otimes \epsilon$. Such is the case if $C=\un$.
\end{rem}

\begin{rem}
The universal coaction of the coend on itself can be expressed in terms of its Hopf algebra structure as follows:
\begin{equation*}
\delta_C=\, \psfrag{C}[Bc][Bc]{\scalebox{.8}{$C$}}\rsdraw{.45}{.95}{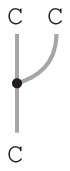}=\;\rsdraw{.45}{.95}{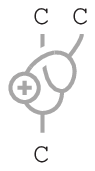}.
\end{equation*}
\end{rem}

\begin{rem}
The coend $C$ has a canonical half braiding $\sigma=\{\sigma_X \}_{X \in \cc}$ defined by
\begin{equation*}
\sigma_C=\, \psfrag{C}[Bc][Bc]{\scalebox{.8}{$C$}}\psfrag{X}[Bc][Bc]{\scalebox{.8}{$X$}}\rsdraw{.45}{.95}{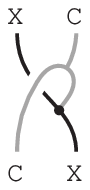} \,: C \otimes X\to X \otimes C.
\end{equation*}
Then  $(C,\sigma)$, endowed with the coproduct and counit of $C$, is a coalgebra in $\zz(\cc)$ which is cocommutative. Indeed its coproduct $\Delta_C$ satisfies $\sigma_C \Delta_C=\Delta_C$.
\end{rem}

\subsection{\Rt matrices: Majid's approach}
Let $A$ be a Hopf algebra in braided category $\bb$, with braiding $\tau$. In \cite{Maj3}, by extending Drinfeld's axioms, Majid
defined  an \Rt matrix for $A$ to be  a convolution-invertible morphism $\mathfrak{r}\co \un \to A \otimes A$ satisfying
\begin{center}
\psfrag{C}[Bc][Bc]{\scalebox{.8}{$C$}} \psfrag{A}[Bc][Bc]{\scalebox{.8}{$A$}}
\psfrag{r}[Bc][Bc]{\scalebox{.9}{$\mathfrak{r}$}}
\rsdraw{.45}{1}{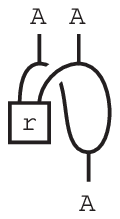} \, $=$ \, \rsdraw{.45}{1}{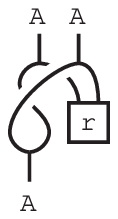}\;, \qquad
\rsdraw{.45}{1}{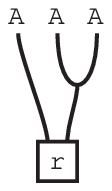} $=$ \rsdraw{.45}{1}{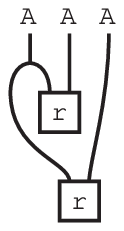}\;, \qquad
\rsdraw{.45}{1}{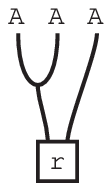} $=$ \rsdraw{.45}{1}{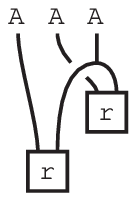}\,.
\end{center}
Here $\mathfrak{r}$ convolution-invertible means that there exists a (necessarily unique) morphism $\mathfrak{r}'\co \un \to A \otimes A$ such that
\begin{center}
\psfrag{A}[Bc][Bc]{\scalebox{.8}{$A$}}
\psfrag{r}[Bc][Bc]{\scalebox{.9}{$\mathfrak{r}$}}
\psfrag{i}[Bl][Bl]{\scalebox{.9}{$\mathfrak{r}'$}}
\rsdraw{.45}{1}{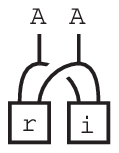} \, $=$ \, \rsdraw{.45}{1}{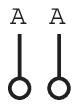} $=$ \, \rsdraw{.45}{1}{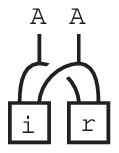}\,.
\end{center}
Note that if $\bb$ is rigid, then $\mathfrak{r}$ is convolution-invertible if and only if it satisfies
\begin{center}
\psfrag{A}[Bc][Bc]{\scalebox{.8}{$A$}}
\psfrag{r}[Bc][Bc]{\scalebox{.9}{$\mathfrak{r}$}}
\rsdraw{.45}{1}{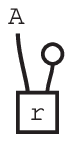} \, $=$ \, \rsdraw{.45}{1}{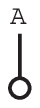} \, $=$ \, \rsdraw{.45}{1}{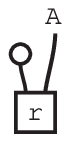} \,.
\end{center}

Majid noticed that such a morphism  $\mathfrak{r}$ does not define a braiding on the category $\bb_A$ of $A$-modules braided, but  only on the full subcategory $\mathcal{O}_A$ of $\bb_A$ whose objects are right $A$-modules $(M,r)$ such that
\begin{center}
\psfrag{X}[Bc][Bc]{\scalebox{.8}{$M$}} \psfrag{A}[Bc][Bc]{\scalebox{.8}{$A$}}
\psfrag{n}[Bc][Bc]{\scalebox{.9}{$r$}}
\rsdraw{.45}{1}{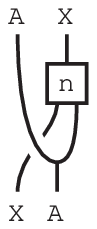} \, $=$ \, \rsdraw{.45}{1}{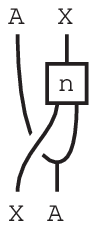} \,.
\end{center}
This braiding on $\mathcal{O}_A$ is given by
\begin{equation*}
c_{(M,r),(N,s)}=\psfrag{X}[Bc][Bc]{\scalebox{.8}{$M$}} \psfrag{Y}[Bc][Bc]{\scalebox{.8}{$N$}}
\psfrag{n}[Bc][Bc]{\scalebox{.9}{$r$}}\psfrag{s}[Bc][Bc]{\scalebox{.9}{$s$}} \psfrag{r}[Bc][Bc]{\scalebox{.9}{$\mathfrak{r}$}}
\rsdraw{.45}{1}{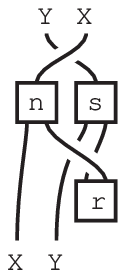}\;.
\end{equation*}
Note that in general $\mathcal{O}_A \neq \bb_A$, and equality occurs only when the Hopf algebra $A$ is transparent, that is, $\tau_{A,X}=\tau_{X,A}^{-1}$
for any object $X$ of $\bb$.

\subsection{\Rt matrices revisited}
Recall that a key feature of \Rt matrices for Hopf algebras over a field is the following (see~\cite{Drin}):
if $H$ is finite-dimensional Hopf algebra $H$ over a field $\kk$, then \Rt matrices for $H$ are in natural bijection with braidings on the category of finite-dimensional $H$-modules. As noted above, this bijective correspondence is lost with Majid's definition.

In \cite{BV3}, using the theories of Hopf monads and coends, we extended the notion of an \Rt matrix to a Hopf algebra $A$ in braided rigid category admitting a coend, so as to preserve this bijective correspondence.

Let $\bb$ be a braided rigid category admitting a coend $C$, and $A$ be a Hopf algebra in $\bb$. An \emph{\Rt matrix} for $A$ is a morphism
\begin{equation*}
\mathfrak{r}\co C \otimes C \to A \otimes A
\end{equation*}
in $\bb$, which satisfies
\begin{center}
\psfrag{C}[Bc][Bc]{\scalebox{.8}{$C$}} \psfrag{A}[Bc][Bc]{\scalebox{.8}{$A$}}
\psfrag{r}[Bc][Bc]{\scalebox{.9}{$\mathfrak{r}$}}
\rsdraw{.45}{1}{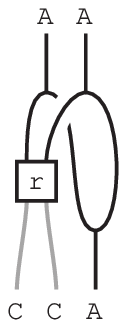} \, $=$ \, \rsdraw{.45}{1}{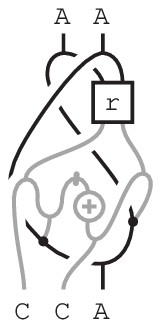}\;, \qquad
\rsdraw{.45}{1}{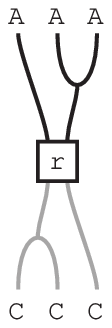} $=$ \rsdraw{.45}{1}{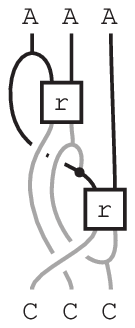}\;, \qquad
\rsdraw{.45}{1}{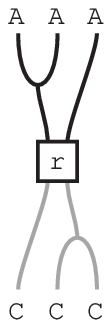} $=$ \rsdraw{.45}{1}{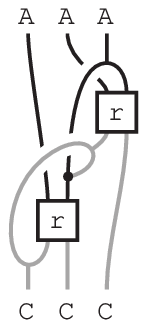}\,,\\[.8em]
\rsdraw{.45}{1}{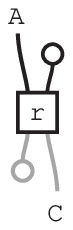} \, $=$ \, \rsdraw{.45}{1}{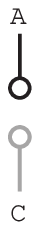} \, $=$ \, \rsdraw{.45}{1}{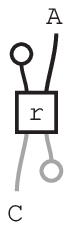} \,.
\end{center}

Note that for finite-dimensional Hopf algebras over a field $\kk$, our definition of an \Rt matrix coincides with Drinfeld's definition, as in that case $C = \kk$.

\begin{thm}[{\cite[Section 8.6]{BV3}}]\label{thm-rmat}
Any \Rt matrix $\mathfrak{r}$ for $A$ defines a braiding~$c$ on $\rmod{\bb}{A}$ as follows: for right $A$-modules $(M,r),(N,s)$,
\begin{equation*}
c_{(M,r),(N,s)}=\psfrag{X}[Bc][Bc]{\scalebox{.8}{$M$}} \psfrag{Y}[Bc][Bc]{\scalebox{.8}{$N$}}
\psfrag{n}[Bc][Bc]{\scalebox{.9}{$r$}}\psfrag{s}[Bc][Bc]{\scalebox{.9}{$s$}} \psfrag{r}[Bc][Bc]{\scalebox{.9}{$\mathfrak{r}$}}
\rsdraw{.45}{1}{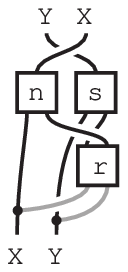}\;.
\end{equation*}
Furthermore, the map $\mathfrak{r} \mapsto c$ is a bijection between \Rt matrices for $A$ and braidings on $\bb_A$
\end{thm}

\begin{rem}\label{rem-rmat}
\Rt matrices also encode braiding on the category $\lmod{\bb}{A}$ of left $A$-modules. Indeed, since
 braidings on $\lmod{\bb}{A}$ are in bijective correspondence with braidings on  $\rmod{\bb}{A}$ (see Remark~\ref{rem-braidleftright}), an \Rt matrix $\mathfrak{r}$ for $A$ defines a braiding $c'$ on  $\lmod{\bb}{A}$ as follows: for left $A$-modules $(M,r),(N,s)$,
\begin{equation*}
c'_{(M,r),(N,s)}=\psfrag{X}[Bc][Bc]{\scalebox{.8}{$M$}} \psfrag{Y}[Bc][Bc]{\scalebox{.8}{$N$}}
\psfrag{n}[Bc][Bc]{\scalebox{.9}{$r$}}\psfrag{s}[Bc][Bc]{\scalebox{.9}{$s$}} \psfrag{r}[Bc][Bc]{\scalebox{.9}{$\mathfrak{r}$}}
\rsdraw{.45}{1}{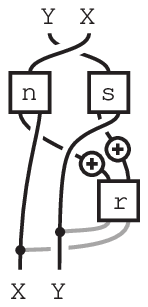}\;.
\end{equation*}
Furthermore, the map $\mathfrak{r} \mapsto c'$  is a bijection between \Rt matrices for $A$ and braidings on $\lmod{\bb}{A}$.
\end{rem}

\subsection{Quasitriangular Hopf algebras}
Let $\bb$ be a braided rigid category admitting a coend.
A \emph{quasitriangular Hopf algebra} in $\bb$ is a Hopf algebra in $\bb$ endowed with an \Rt matrix.

By Theorem~\ref{thm-rmat} and Remark~\ref{rem-rmat}, if $A$ is a quasitriangular Hopf algebra in $\bb$, then the rigid categories  $\rmod{\bb}{A}$ and $\lmod{\bb}{A}$ are braided.

\begin{rem}\label{rem-rmatmirror}
If $A$ is a quasitriangular Hopf algebra in $\bb$, then  $\lmod{\bb}{A}$ and $\rmod{\bb}{A}$ are isomorphic as braided categories. Indeed, the monoidal functor $(1_{\lmod{\bb}{A}},c', \id_{\un}) \co {\lmod{\bb}{A}}^{\otimes \opp} \to {\lmod{\bb}{A}}$ is a braided isomorphism (where $c'$ is the braiding of $\lmod{\bb}{A}$), and
by construction the monoidal isomorphism $F_A \co (\lmod{\bb}{A})^{\otimes\opp} \to \rmod{\bb}{A}$ of Remark~\ref{rem-braidleftright} is braided.
\end{rem}

\begin{exa}\label{ex-Cqt}
The coend $C$ of $\bb$ is a quasitriangular Hopf algebra in $\bb$ with \Rt matrix
\begin{equation*}
\mathfrak{r}=\, \psfrag{C}[Bc][Bc]{\scalebox{.8}{$C$}}\rsdraw{.45}{.95}{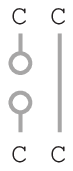} \,,
\end{equation*}
and so the category $\bb_C$ of right $C$-modules is braided.
For any right $C$-module $(M,r)$ and any $C$-linear morphism $f$, set  $I(M,r)=(M,\sigma)$ and $I(f)=f$ with
\begin{equation*}
\sigma_X=\,  \psfrag{r}[Bc][Bc]{\scalebox{.8}{$r$}}\psfrag{C}[Bc][Bc]{\scalebox{.8}{$M$}} \psfrag{X}[Bc][Bc]{\scalebox{.8}{$X$}}\rsdraw{.45}{.95}{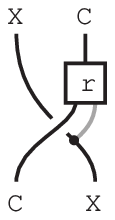} \;.
\end{equation*}
This defines a functor $I\co \bb_C \to \zz(\bb)$ which is a braided strict monoidal isomorphism.
\end{exa}

\subsection{The double of a Hopf algebra}\label{sect-DA}
Let $\bb$ be a braided rigid category admitting a coend $C$, and let $A$ be a Hopf algebra in $\bb$. Set
$$
D(A)=A \otimes \ldual{A} \otimes C
$$
and define the product $m_{D(A)}$, the unit $u_{D(A)}$, the coproduct $\Delta_{D(A)}$, the counit $\varepsilon_{D(A)}$, the antipode $S_{D(A)}$, and the \Rt matrix $\mathfrak{r}_{D(A)}$ as in Figure~\ref{fig-DA}.
\begin{figure}[t]
\begin{center}
\psfrag{A}[Bc][Bc]{\scalebox{.8}{$A$}}
\psfrag{B}[Bc][Bc]{\scalebox{.8}{$\ldual{A}$}}
\psfrag{D}[Br][Br]{\scalebox{.8}{$\ldual{A}$}}
\psfrag{C}[Bc][Bc]{\scalebox{.8}{$C$}}
$m_{D(A)}$ \, = \, \rsdraw{.45}{.9}{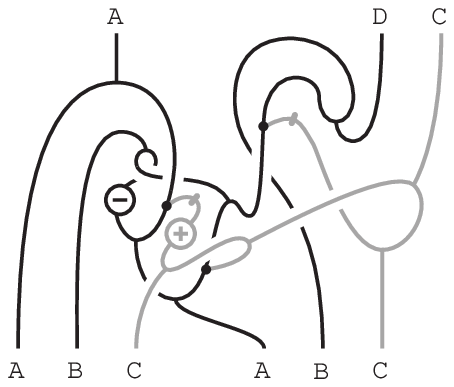}\,,  \\[1em]
\psfrag{A}[Bc][Bc]{\scalebox{.8}{$A$}}
\psfrag{B}[Bc][Bc]{\scalebox{.8}{$\ldual{A}$}}
\psfrag{F}[Br][Br]{\scalebox{.8}{$\ldual{A}$}}
\psfrag{C}[Bc][Bc]{\scalebox{.8}{$C$}}
$\Delta_{D(A)}$ \, = \, \rsdraw{.45}{.9}{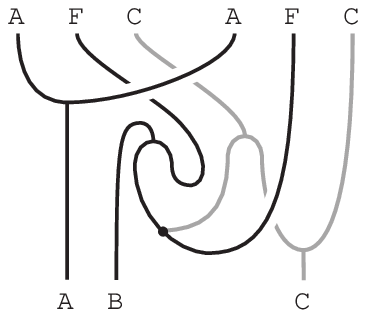} \,,\\[1em]
\psfrag{A}[Bc][Bc]{\scalebox{.8}{$A$}}
\psfrag{B}[Br][Br]{\scalebox{.8}{$\ldual{A}$}}
\psfrag{C}[Bc][Bc]{\scalebox{.8}{$C$}}
$u_{D(A)}$ \, = \, \rsdraw{.45}{.9}{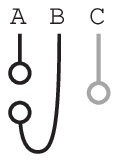}\;, \qquad
$\varepsilon_{D(A)}$ \, = \, \rsdraw{.45}{.9}{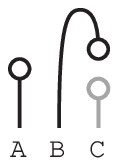}\;, \\[1em]
\psfrag{A}[Bc][Bc]{\scalebox{.8}{$A$}}
\psfrag{B}[Bc][Bc]{\scalebox{.8}{$\ldual{A}$}}
\psfrag{F}[Br][Br]{\scalebox{.8}{$\ldual{A}$}}
\psfrag{C}[Bc][Bc]{\scalebox{.8}{$C$}}
$S_{D(A)}$ \, = \, \rsdraw{.45}{.9}{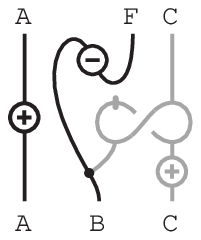} \,, \qquad
\psfrag{A}[Bc][Bc]{\scalebox{.8}{$A$}}
\psfrag{B}[Br][Br]{\scalebox{.8}{$\ldual{A}$}}
\psfrag{C}[Bc][Bc]{\scalebox{.8}{$C$}}
$\mathfrak{r}_{D(A)}$ \, = \, \rsdraw{.45}{.9}{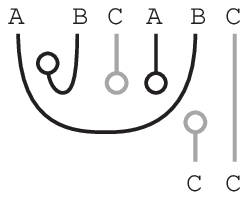} \,.
\end{center}
\caption{Structural morphisms of the double $D(A)$ of $A$}
\label{fig-DA}
\end{figure}

\begin{thm}[{\cite[Theorem~8.13]{BV3}}]\label{thm-double-DA} In the above notation,
$D(A)=A \otimes \ldual{A} \otimes C$ is a quasitriangular Hopf algebra in $\bb$, and we have the following isomorphisms of braided categories:
\begin{equation*}
\zz(\rmod{\bb}{A}) \simeq \rmod{\bb}{{D(A)}}\simeq\lmod{\bb}{{D(A)}} \simeq \overline{\zz(\lmod{\bb}{A})}.
\end{equation*}
\end{thm}

\begin{rem}
When $\bb=\vect_\kk$ is the category of finite-dimensional vector
spaces over a field $\Bbbk$, whose coend is $\kk$, we recover the usual Drinfeld double
and the interpretation of its category of modules in terms of the
center. More precisely,  let $H$ be a finite-dimensional Hopf algebra
over $\Bbbk$ and $(e_i)$ be a basis of $H$ with dual basis $(e^i)$. Then $D(H)=H \otimes (H^*)^\mathrm{cop}$ is a quasitriangular Hopf algebra over $\Bbbk$, with \Rt matrix $\mathfrak{r}=\sum_i   e_i \otimes \varepsilon \otimes 1_H \otimes e_i$, such that
\begin{equation*}
 \zz\bigl((\vect_\kk)_H\bigr)\simeq (\vect_\kk)_{D(H)} \simeq \lmod{(\vect_\kk)}{D(H)} \simeq \overline{\zz\bigl(\lmod{(\vect_\kk)}{H}\bigr)}
\end{equation*}
as braided categories.
\end{rem}

\begin{rem}
The coend $C$ is nothing but the quasitriangular Hopf algebra $D(\un)$ (see Example~\ref{ex-Cqt}).
\end{rem}

\subsection{Yetter-Drinfeld modules}\label{sect-YDs}

%
%

Let $A$ be a bialgebra in a braided category $\bb$ with braiding $\tau$.
A  \emph{(left-left) Yetter-Drinfeld module} of $A$ in $\bb$ is a an object $M$ of~$\bb$ endowed with a left $A$-action $r \co A \otimes X \to M$ and a left $A$-coaction
$\delta \co A \otimes X \to M$, such that
$$
\psfrag{V}[Bc][Bc]{\scalebox{.8}{$M$}}
\psfrag{N}[Bc][Bc]{\scalebox{.8}{$N$}}
\psfrag{A}[Bc][Bc]{\scalebox{.8}{$A$}}
\psfrag{B}[Br][Br]{\scalebox{.8}{$\ldual{A}$}}
\psfrag{d}[Bc][Bc]{\scalebox{1}{$\delta$}}
\psfrag{r}[Bc][Bc]{\scalebox{1}{$r$}}
\rsdraw{.45}{.9}{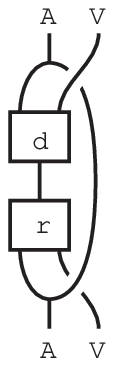} \, = \, \rsdraw{.45}{.9}{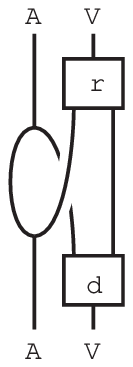}\;.
$$
These Yetter-Drinfeld modules are the objects of a category $\YD{A}(\bb)$, whose morphisms are morphisms in $\bb$ which are $A$\ti linear and
$A$\ti colinear. This is a braided category. Its monoidal product is $(M,r,\delta) \otimes (M',r',\delta')=(M \otimes M', r'', \delta'')$ where
$$
\psfrag{A}[Bc][Bc]{\scalebox{.8}{$A$}}
\psfrag{U}[Bc][Bc]{\scalebox{.8}{$M$}}
\psfrag{r}[Bc][Bc]{\scalebox{1}{$r$}}
\psfrag{s}[Bc][Bc]{\scalebox{1}{$r'$}}
\psfrag{Y}[Bc][Bc]{\scalebox{.8}{$M'$}}
r''\, = \, \rsdraw{.45}{.9}{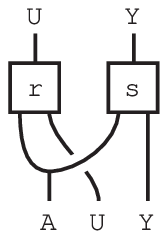}
 \quad \text{and} \quad \delta''\, = \,
\psfrag{r}[Bc][Bc]{\scalebox{1}{$\delta$}}
\psfrag{s}[Bc][Bc]{\scalebox{1}{$\delta'$}}
 \rsdraw{.45}{.9}{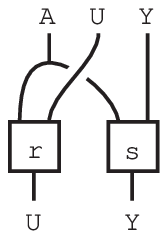} \;\,,
$$
its monoidal unit is $(\un,\varepsilon,u)$ where $u$ and $\varepsilon$ are the unit and counit of $A$, and its braiding is
$$
\psfrag{A}[Bc][Bc]{\scalebox{.8}{$A$}}
\psfrag{U}[Bc][Bc]{\scalebox{.8}{$M$}}
\psfrag{r}[Bc][Bc]{\scalebox{1}{$\delta$}}
\psfrag{s}[Bc][Bc]{\scalebox{1}{$r'$}}
\psfrag{Y}[Bc][Bc]{\scalebox{.8}{$M'$}}
c_{(M,r,\delta),(M',r',\delta')}\, = \, \rsdraw{.45}{.9}{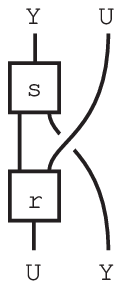} \,.
$$

Now let $\cc$ be a monoidal category, and let $\AT = (A,\sigma)$ be a bialgebra in the center $\zz(\cc)$ of $\cc$. The half braiding $\sigma$ defines on the category $\lmod{\cc}{A}$ of
left $A$-modules a monoidal structure, with monoidal product $(M,r) \otimes (N,s)=(M \otimes N, \omega)$ where
$$
\omega=\psfrag{A}[Bc][Bc]{\scalebox{.8}{$A$}} \psfrag{U}[Bc][Bc]{\scalebox{.8}{$M$}} \psfrag{r}[Bc][Bc]{\scalebox{1}{$r$}}  \psfrag{s}[Bc][Bc]{\scalebox{1}{$s$}} \psfrag{Y}[Bc][Bc]{\scalebox{.8}{$N$}}\psfrag{X}[Bc][Bc]{\scalebox{1}{$\sigma_M$}} \rsdraw{.45}{.9}{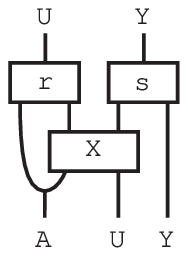} \, ,
$$
and monoidal unit $(\un,\varepsilon)$ where $\varepsilon$ is the counit of $\AT$. We denote this monoidal category by $\lmod{\cc}{\AT}$.  Note that if $A$ is a bialgebra in a braided category $\bb$ with braiding $\tau$, then $\AT=(A,\tau_{A,-})$ is a Hopf algebra in $\zz(\bb)$ and $\lmod{\bb}{A}=\lmod{\bb}{\AT}$ as monoidal categories.

\begin{prop}\label{thm-YD}
Let $\cc$ be a monoidal category and let $\AT$ be a Hopf algebra in the center $\zz(\cc)$ of $\cc$.  The assignment
\begin{equation*}
\left \{
\begin{array}{ccc}
\YD{\AT}(\zz(\cc)) & \to & \zz(\lmod{\cc}{\AT}) \\
\bigl((M,\sigma),r,\delta\bigr) & \mapsto & \bigl((M,r),\gamma\bigr)\\
f & \mapsto & f
\end{array}\right.
\end{equation*}
where
$$
\psfrag{A}[Bc][Bc]{\scalebox{.8}{$A$}} \psfrag{U}[Bc][Bc]{\scalebox{.8}{$M$}} \psfrag{r}[Bc][Bc]{\scalebox{1}{$\delta$}}  \psfrag{s}[Bc][Bc]{\scalebox{1}{$\sigma_N$}} \psfrag{t}[Bc][Bc]{\scalebox{1}{$t$}} \psfrag{Y}[Bc][Bc]{\scalebox{.8}{$N$}}
 \gamma_{(N,t)}
\, = \, \rsdraw{.45}{.9}{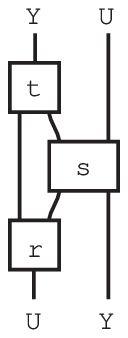} \; \,,
$$
is a braided strict monoidal isomorphism.
\end{prop}

\begin{proof}
This is proved by direct computation, the inverse functor being given by $\sigma_X = \gamma_{(X, \varepsilon \otimes \id_X)}$ and $\delta = \gamma_{(A,m)}(\id_M \otimes u)$ where $m$, $u$, and $\varepsilon$ are the product, unit, and counit of  $\AT$. The invertibility of $\gamma$ results from the existence of the antipode of~$\AT$.
\end{proof}

\begin{rem}
The isomorphism of Proposition~\ref{thm-YD} still holds for bialgebras replacing centers with lax centers.
\end{rem}

%

\begin{rem}\label{rem-ydgtr}
Let $A$ be a Hopf algebra in a braided category $\bb$ with braiding $\tau$. As seen above, $\AT=(A,\tau_{A,-})$ is a Hopf algebra in $\zz(\bb)$. Applying Proposition~\ref{thm-YD}, we see that $\YD{\AT}(\zz(\bb))\simeq \zz(\lmod{\bb}{A})$ as braided categories.
\end{rem}

\subsection{A non-representability result}\label{sect-non-rep-res}
In Section~\ref{sect-DA}, we have seen that, given a Hopf algebra $A$ in a braided rigid category $\bb$ admitting a coend, the center of the category of $A$\ti modules
can be described as the category of modules of a (quasitriangular) Hopf algebra in $\bb$.

In this section, we give an example of a Hopf algebra $A$ in a braided category $\bb$ such that the center of the category of $A$\ti modules is not a category of modules of the form  $\leftidx{_\mathbb{B}}{\zz(\bb)}{}$ for any Hopf algebra $\mathbb{B}$ in $\zz(\bb)$ nor in the center of $\zz(\bb)$.

Let $n\geq 2$ be an integer and denote by $\bb_n$ the category of finite dimensional $(\mathbb{Z}/n\mathbb{Z})$-graded \kt vector spaces, where $\kk$ is a field containing a primitive $n^\text{th}$ root of unity $q$. Then $\bb_n$ is a braided rigid category with braiding $\tau$ defined as follows: for $U,V$ in $\bb_n$ and homogeneous $u \in U$, $v\in V$,
$$
\tau_{U,V}(u \otimes v)=q^{|u||v|}\, v \otimes u.
$$
Consider the algebra $A_n=\kk[X]/(X^n)$ as an object of $\bb_n$, with the graduation given by the polynomial degree.
By \cite[Example 3.1]{BKLT}, $A_n$ is an Hopf algebra in $\bb_n$. It is generated by the class $x$ of $X$. Its counit, coproduct, and antipode
are given by $\varepsilon(x^m)=\delta_{m,0}$,
$$
\Delta(x^m)=\sum_{k=0}^{m} \binom{m}{k}_{\!\!q} \,x^k \otimes x^{m-k}, \quad \text{and} \quad
  S(x^m)=(-1)^m q^{m(m-1)/2}\,x^m.
$$

Let
$
\uu \co\zz(\lmod{\bb_n}{A_n}) \to \zz(\bb_n)
$
be the functor defined on objects by  $$\uu((M,r),\gamma)=(M, \sigma=\{\sigma_U=\gamma_{(U, \id_U \otimes \varepsilon)}\}_{U \in \bb_n})$$ and on morphisms by $\uu(f)=f$.

\begin{prop}\label{prop-rere}
The functor $\uu \co\zz(\lmod{\bb_n}{A_n}) \to \zz(\bb_n)$ is not essentially surjective.
\end{prop}
\begin{proof}
Set $\AT_n=(A_n,\tau_{A_n,-})$. By Remark~\ref{rem-ydgtr}, it is enough to show that the forgetful functor $$\YD{\AT_n}(\zz(\bb_n)) \to \zz(\bb_n)$$ is not essentially surjective. Let $\chi$ be a non-trivial character of $\mathbb{Z}/n\mathbb{Z}$. For $U \in \bb_n$, define $\sigma_U \co  U \to U$ by $\sigma(u)=\chi(|u|)\, u$ for homogeneous $u \in U$. Then $\kk_\chi=(\kk,\sigma)$ is an object of $\zz(\bb_n)$. Assume there exists a Yetter-Drinfeld module $M=((M,\rho),r,\delta)$ over $\AT_n$  such that $F(M)$ is isomorphic to $\kk_\chi$. Without loss of generality, we assume that $(M,\rho)=\kk_\chi$. Now all $A_n$\ti actions or coactions on $\kk$ are trivial, that is, $r=\varepsilon \otimes \id_\kk$ and $\delta=u \otimes \id_\kk$ where $u$ and $\varepsilon$ are the unit and counit of $A_n$. Indeed, this follows the fact the degree 0 part of $A_n$ is $\kk$.  Now the Yetter-Drinfeld compatibility axiom implies that $\sigma_{A_n}=\id_{A_n}$. Hence $\sigma_U = \id_U$  for any $U \in \bb_n$ because $A_n$ generates $\bb_n$. This contradicts the fact that $\chi$ is non-trivial, and proves the proposition.
\end{proof}

\begin{cor}\label{prop-EX}
There exists no pair $(\mathbb{B},F)$, where $\mathbb{B}$  is a Hopf algebra in  $\zz(\bb_n)$ or its center $\zz(\zz(\bb_n))$, and $F \co \lmod{\zz(\bb_n)}{\mathbb{B}} \to \zz(\lmod{\bb_n}{A_n})$ is monoidal equivalence, such that
the diagram
$$\xymatrix@C=1pc{\lmod{\zz(\bb_n)}{\mathbb{B}} \ar[rr]^F\ar[rd]_{U} &\ar@{}[d]|(.4){\circlearrowright}& \zz(\lmod{\bb_n}{A_n}) \ar[ld]^{\uu}\\ & \zz(\bb_n)}$$
commutes up to monoidal natural isomorphism, where $U$ is the forgetful functor.
\end{cor}
\begin{proof}
Observe that if $\mathbb{B}$ is a Hopf algebra in the center of a monoidal category~$\cc$, then the forgetful functor $\lmod{\cc}{\mathbb{B}} \to \cc$ admits a strict monoidal section given by $X \mapsto (X, \varepsilon \otimes \id_X)$, where $\varepsilon$ is the counit of $\mathbb{B}$, and it is therefore essentially surjective. This general fact together with Proposition~\ref{prop-rere} proves the corollary.
\end{proof}

\section{Hopf monads}\label{sect-Hopf-Monads-resume}

In this section we recall the definition of Hopf monads, and we list several results. See \cite{BV2,BLV} for a detailed treatment.

\subsection{Comonoidal functors}\label{sect-comonofunctor}
Let $(\cc,\otimes,\un)$ and $(\dd, \otimes, \un)$ be two monoidal categories.
A \emph{comonoidal functor} (also called \emph{opmonoidal functor}) from $\cc$ to
$\dd$ is a triple $(F,F_2,F_0)$, where $F\co \cc \to \dd$ is a functor, $F_2\co F \otimes \to F\otimes F$ is a natural
transformation, and $F_0\co F(\un) \to \un$ is a morphism in $\dd$, such that:
\begin{align*}
& \bigl(\id_{F(X)} \otimes F_2(Y,Z)\bigr) F_2(X,Y \otimes Z)= \bigl(F_2(X,Y) \otimes \id_{F(Z)}\bigr) F_2(X \otimes Y, Z) ;\\
& (\id_{F(X)} \otimes F_0) F_2(X,\un)=\id_{F(X)}=(F_0 \otimes \id_{F(X)}) F_2(\un,X) ;
\end{align*}
for all objects $X,Y,Z$ of $\cc$.

A comonoidal functor $(F,F_2,F_0)$ is said to be \emph{strong} (resp.\@ \emph{strict}) if $F_2$ and $F_0$ are
isomorphisms (resp.\@ identities). In that case, $(F,F^{-1}_2,F^{-1}_0)$ is a strong (resp. strict) monoidal functor.

A natural transformation $\varphi\co F \to G$ between comonoidal functors is \emph{comonoidal} if it satisfies:
\begin{equation*}
G_2(X,Y) \varphi_{X \otimes Y}= (\varphi_X \otimes \varphi_Y) F_2(X,Y)\quad \text{and} \quad G_0 \varphi_\un= F_0.
\end{equation*}

Note that the notions of comonoidal functor and comonoidal natural transformation are dual to the notions of monoidal functor and monoidal natural transformation (see Section~\ref{sect-monofunctor}).

\subsection{Hopf monads and their modules}\label{sect-Hopf-monoads}
Let $\cc$ be a category.
A \emph{monad} on $\cc$  is a monoid in the category of endofunctors of $\cc$, that
is, a triple $(T,\mu,\eta)$ consisting of a functor $T\co \cc \to
\cc$ and two natural transformations $$\mu=\{\mu_X\co T^2(X) \to
T(X)\}_{X \in \cc}\quad {\text {and}} \quad  \eta=\{\eta_X\co X \to
T(X)\}_{X \in \cc}$$  called the \emph{product} and the \emph{unit}
of $T$, such that for any object $X$ of $\cc$, $$\mu_X
T(\mu_X)=\mu_X\mu_{T(X)} \quad {\text {and}} \quad
\mu_X\eta_{T(X)}=\id_{T(X)}=\mu_X T(\eta_X).$$
Given a monad $T=(T, \mu, \eta)$ on $\cc$, a  $T$\ti module in
$\cc$ is a pair $(M,r)$ where $M$ is an object of $\cc$ and $r\co T(M) \to M$ is a
morphism in $\cc$ such that $r T(r)= r \mu_M$ and $r \eta_M= \id_M$.
A morphism from a $T$\ti module  $(M,r)$ to a $T$\ti module $(N,s)$ is a morphism $f \co M \to N$ in $\cc$ such that $f r=s T(f)$.
This defines the {\it category  $\cc^T$  of $T$-modules in $\cc$} with composition induced by that in $\cc$.
We
define a forgetful functor  $U_T\co\cc^T \to \cc$    by $U_T(M,r)=M$ and
$U_T(f)=f$.  

Let $\cc$ be a monoidal category. A \emph{bimonad} on    $\cc$ is a monoid in the
category of  comonoidal endofunctors of $\cc$. In other words, a bimonad on $\cc$ is a
monad $(T,\mu,\eta)$ on $\cc$ such that the  functor $T\co \cc \to
\cc$ and the natural transformations $\mu$ and $\eta$ are
comonoidal, that is, $T$ comes equipped with a natural transformation $ T_2=\{T_2(X,Y) \co  T(X \otimes Y)\to T(X)
\otimes T(Y)\}_{X,Y \in \cc} $ and a morphism $T_0\co T(\un) \to \un$ such that
\begin{align*}
& \bigl(\id_{T(X)} \otimes T_2(Y,Z)\bigr) T_2(X,Y \otimes Z)= \bigl(T_2(X,Y) \otimes \id_{T(Z)}\bigr) T_2(X \otimes Y, Z) ;\\
& (\id_{T(X)} \otimes T_0) T_2(X,\un)=\id_{T(X)}=(T_0 \otimes \id_{T(X)}) T_2(\un,X) ;\\
& T_2(X,Y)\mu_{X \otimes Y}=(\mu_X \otimes \mu_Y) T_2(T(X),T(Y))T(T_2(X,Y));\\
& T_2(X,Y)\eta_{X \otimes Y}=\eta_X \otimes \eta_Y.
\end{align*}
For any bimonad $T$ on $\cc$, the category  of $T$\ti modules
$\cc^T$  has a monoidal structure with unit object $(\un,T_0)$ and
with tensor product
$$(M,r) \otimes (N,s)=\bigl(M \otimes N, (r \otimes s) \,
T_2(M,N)\bigr). $$ Note that the forgetful functor $U_T\co \cc^T \to
\cc$ is strict monoidal.

A \emph{quasitriangular bimonad} on $\cc$ is a bimonad $T$ on $\cc$  equipped with an \Rt matrix, that is, a natural transformation $$R=\{R_{X,Y}\co
X \otimes Y \to T(Y) \otimes T(X)\}_{X,Y \in \cc}$$ satisfying appropriate axioms which ensure
that  the natural transformation $ \tau=\{\tau_{(M,r),(N,s)}\}_{(M,r), (N,s) \in \cc^T}$ defined by
$$
\tau_{(M,r),(N,s)}=(s \otimes r) R_{M,N}\co (M,r) \otimes (N,s) \to (N,s) \otimes (M,r)
$$
form a braiding in the category $\cc^T$ of $T$-modules, see~\cite{BV2}.

Given a  bimonad $(T,\mu,\eta)$ on $\cc$ and objects $X, Y\in \cc$,
one defines  the \emph{left fusion operator}
$$H^l_{X,Y} =(T(X) \otimes \mu_Y)T_2(X,T(Y)) \colon T(X\otimes T(Y)) \to T(X) \otimes
T(Y)$$
 and the \emph{right fusion operator}
$$
 H^r_{X,Y}=(\mu_X \otimes
T(Y))T_2(T(X),Y)\colon T(T(X)\otimes Y) \to T(X)\otimes T(Y).$$
A \emph{Hopf monad} on   $\cc$ is a bimonad on $\cc$
whose left and right fusion  operators  are isomorphisms for all objects $X, Y$ of $\cc$.
When $\cc$ is a rigid category,    a bimonad $T$ on
$\cc$ is a Hopf monad if and only if the category  $\cc^T$ is rigid.
The  structure of a rigid category in $\cc^T$ can then be
 encoded in terms of natural transformations
$$s^l=\{s^l_X\co T(\leftidx{^\vee}{T}{}(X)) \to \leftidx{^\vee}{X}{}\}_{X \in \cc}
\quad {\text {and}} \quad s^r=\{s^r_X\co T(T(X)^\vee) \to
X^\vee\}_{X \in \cc}$$ called the \emph{left and right antipodes}.
They are computed from the   fusion operators:
\begin{align*}
& s^l_X= \bigl(T_0T(\lev_{T(X)})(H^l_{\leftidx{^\vee}{T}{}(X),X})^{-1} \otimes \leftidx{^\vee}{\eta}{_X}\bigr)
\bigl(\id_{T(\leftidx{^\vee}{T}{}(X))} \otimes \lcoev_{T(X)}\bigr);\\
& s^r_X= \bigl(\eta_X^\vee \otimes T_0T(\rev_{T(X)})(H^r_{X,T(X)^\vee})^{-1}\bigr)
\bigl(\rcoev_{T(X)} \otimes \id_{T(T(X)^\vee)}\bigr).
\end{align*}
The  left and right duals of any $T$\ti module $(M,r)$ are then defined
  by
$$
\leftidx{^\vee}{(}{} M,r)=(\leftidx{^\vee}{M}{}, s^l_M T(\leftidx{^\vee}{r}{})  \quad \text{and} \quad(M,r)^\vee=(M^\vee, s^r_M T(r^\vee).
$$

\subsection{Centralizers}\label{sect-centralizers}
Let $T$ be a Hopf monad on a rigid category $\cc$. We say that $T$ is \emph{centralizable} if, for any object $X$ of $\cc$, the coend $$Z_T(X)=\int^{Y \in \cc} \ldual{T(Y)}\otimes X \otimes Y$$
exists (see \cite{BV3}).
In that case, the assignment  $X \mapsto Z_T(X)$ is a Hopf monad on~$\cc$, called the centralizer of $T$ and denoted by $Z_T$.
In particular, we say that $\cc$ is \emph{centralizable} if the identity functor $\id_\cc$ is centralizable. In that case, its centralizer $Z = Z_{1_\cc}$ is a quasitriangular Hopf monad on $\cc$, called the \emph{centralizer of $\cc$}. Moreover, there is a canonical isomorphism of braided categories
$\zz(\cc)\simeq \cc^Z$, see \cite{BV3}.


\subsection{Hopf algebras of the center define Hopf monads}\label{sect-rep-cent}

 Let $\cc$ be a monoidal category.
Any  bialgebra $\AT = (A,\sigma)$ of the center $\zz(\cc)$ of $\cc$  gives rise to a bimonad on $\cc$,
denoted by $A \otimes_\sigma ?$. It is defined by $A \otimes ?$ as a functor, with the monad structure defined by
\begin{center}
$\mu_X= m \otimes X =\psfrag{A}[Bc][Bc]{\scalebox{.8}{$A$}} \psfrag{X}[Bc][Bc]{\scalebox{.8}{$X$}}
  \rsdraw{.45}{.9}{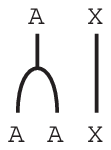} \quad\text{and}\quad \eta_X=u\otimes X=\rsdraw{.45}{.9}{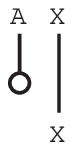},$
\end{center}
where
$m$ and $u$ are the product and unit of $A$, and endowed with the comonoidal structure:
\begin{center}
$(A \otimes_\sigma ?)_2(X,Y)=(A \otimes \sigma_X )(\Delta \otimes X)\otimes Y=\psfrag{A}[Bc][Bc]{\scalebox{.8}{$A$}} \psfrag{U}[Bc][Bc]{\scalebox{.8}{$X$}} \psfrag{Y}[Bc][Bc]{\scalebox{.8}{$Y$}}\psfrag{X}[Bc][Bc]{\scalebox{1}{$\sigma_X$}} \rsdraw{.45}{.9}{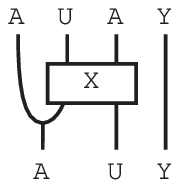}
\,, \quad
(A \otimes_\sigma ?)_0=\varepsilon=\psfrag{A}[Bc][Bc]{\scalebox{.8}{$A$}}\rsdraw{.45}{.9}{epsA},$
\end{center}
where $\Delta$ and $\varepsilon$ denote the coproduct and counit of $(A,\sigma)$.

The monoidal category  $\cc^{A \otimes_\sigma ?}$ is the monoidal category $_{\AT}\cc$ encountered in Section~\ref{sect-YDs}.


Let $(A,\sigma)$ be a bialgebra of $\zz(\cc)$.
The left and right fusion operators of the bimonad $A \otimes_\sigma ?$ are
\begin{align*}
&H^l_{X,Y}=(A \otimes X \otimes m)(A \otimes \sigma_X \otimes A)(\Delta \otimes X \otimes A) \otimes Y= \psfrag{A}[Bc][Bc]{\scalebox{.8}{$A$}} \psfrag{U}[Bc][Bc]{\scalebox{.8}{$X$}} \psfrag{Y}[Bc][Bc]{\scalebox{.8}{$Y$}}\psfrag{X}[Bc][Bc]{\scalebox{1}{$\sigma_X$}} \rsdraw{.45}{.9}{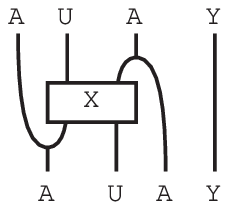}\,,\\
&H^r_{X,Y}=(m \otimes X  \otimes A)(A \otimes \sigma_{A \otimes X})(\Delta \otimes A \otimes X) \otimes Y= \psfrag{A}[Bc][Bc]{\scalebox{.8}{$A$}} \psfrag{U}[Bc][Bc]{\scalebox{.8}{$X$}} \psfrag{Y}[Bc][Bc]{\scalebox{.8}{$Y$}}\psfrag{X}[Bc][Bc]{\scalebox{1}{$\sigma_{A\otimes X}$}} \rsdraw{.45}{.9}{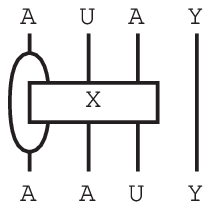}\,.
\end{align*}
So, by Remark~\ref{rem-HA-fusion}, the bimonad $A \otimes_\sigma ?$ is a Hopf monad if and only if $(A,\sigma)$ is
a Hopf algebra.

\subsection{Characterization of Hopf monads representable by Hopf algebras}\label{sect-char-rep}
Let $\cc$ be a monoidal category. A \emph{Hopf monad morphism} $f\co T \to T'$ between two Hopf monads $(T,\mu,\eta)$ and $(T',\mu',\eta')$ on $\cc$ is a comonoidal natural transformation $f=\{f_X \co T(X) \to T'(X)\}_{X \in \cc}$ such that, for any $X \in \Ob(\cc)$,
$$
f_X \mu_X= \mu'_X f_{T'(X)}T(f_X) \quad \text{and} \quad f_X \eta_X=\eta'_X.
$$

Let $\cc$ be a monoidal category. A Hopf monad $T$ on $\cc$ is \emph{augmented} if it is endowed with an \emph{augmentation}, that is, a Hopf monad morphism $e\co T \to 1_\cc$.

Augmented Hopf monads on $\cc$ form a category
$\HopfMon(\cc)/1_\cc$, whose objects are augmented Hopf monads on $\cc$, and morphisms between two augmented Hopf monads $(T,e)$ and $(T',e')$ are morphisms of Hopf monads $f\co T \to T'$ such that $e'f=e$.

If $(A,\sigma)$ is a Hopf algebra of the center $\zz(\cc)$ of $\cc$,   the Hopf monad $A\otimes_\sigma ?$  is augmented, with augmentation $e= \varepsilon \otimes ? \co A \otimes_\sigma ? \to 1_\cc$, where $\varepsilon$ is the counit of $(A,\sigma)$.  Hence a functor
$$\mathfrak{R} \co \left\{ \begin{array}{ccc} \HopfAlg(\zz(\cc)) & \to & \HopfMon(\cc)/1_\cc \\ (A,\sigma) & \mapsto &(A\otimes_\sigma ?,\varepsilon \otimes ?)\end{array} \right.$$

\begin{thm}[{\cite[Theorem~5.7]{BLV}}]\label{thm-rep-HopfMon}
The functor $\mathfrak{R}$ is an equivalence of categories.
\end{thm}

In other words, Hopf monads representable by Hopf algebras of the center are nothing but augmented Hopf monads. Not all Hopf monads are of this kind:

\begin{rem}
Let $\cc$ be a centralizable rigid category, and let $Z$ be its centralizer, which is a quasitriangular Hopf monad on $\cc$, see Section~\ref{sect-centralizers}. Then augmentations of $Z$ are in one-to-one correspondence with braidings on $\cc$. In particular if $\cc$ is not braided, then $Z$ is not representable by a Hopf algebra of the center of $\cc$. For example,
let $\cc=G\ti \vect$ be the category of finite-dimensional $G$\ti graded vector spaces over a field $\kk$ for some finite group $G$. It is centralizable, and its centralizer is representable by a Hopf algebra of the center of $\cc$ if and only if $G$ is abelian (see~\cite[Remark 9.2]{BV3}).
\end{rem}


Hopf monads on a braided category $\bb$ which are representable by Hopf algebras in $\bb$ can also be characterized as follows:

\begin{cor}[{\cite[Corollary5.9]{BLV}}]\label{cor-rep-braided}
Let $T$ be a Hopf monad on a braided category $\bb$. Then $T$ is isomorphic to the Hopf monad $A \otimes ?$ for some Hopf algebra $A$ in $\bb$
if and only if it is endowed with an augmentation $e \co T \to 1_\cc$ compatible with the braiding $\tau$ of $\bb$ in the following sense: for any object $X$ of $\bb$,
$$(e_X \otimes T\un)T_2(X,\un)=(e_X \otimes T\un)\tau_{T\un,TX}T_2(\un,X).$$
\end{cor}

\section{The central double of a Hopf algebra}\label{sect-central-double}
Let $A$ be a Hopf algebra in a  braided rigid  category $\bb$. Remark that any object $X$ of $\bb$ has a trivial right $A$-action given by $\id_X \otimes \varepsilon \co X \otimes A \to X$, where $\varepsilon$ is the counit of $A$. This defines a functor
$$
\uu \co \zz(\bb_A) \to \zz(\bb),
$$
by setting $\uu((M,r),\gamma)=(M, \sigma=\{\sigma_X=\gamma_{(X, \id_X \otimes \varepsilon)}\}_{X \in \bb})$ on objects and $\uu(f)=f$ on morphisms. Then $\uu$ is a strict monoidal functor. In this section, we prove that $\uu$ is monadic and we explicit its associated quasitriangular Hopf monad.

For any object $(M,\sigma)$ of $\zz(\bb)$, set $$d_A(M,\sigma)=\bigl(M \otimes A \otimes \ldual{A},\varsigma=\{\varsigma_X\}_{X \in \bb}\bigr)
\quad \text{with} \quad
\psfrag{A}[Bc][Bc]{\scalebox{.8}{$A$}}
\psfrag{B}[Br][Br]{\scalebox{.8}{$\ldual{A}$}}
\psfrag{C}[Bc][Bc]{\scalebox{.8}{$C$}}
\psfrag{V}[Bc][Bc]{\scalebox{.8}{$M$}}
\psfrag{X}[Bc][Bc]{\scalebox{.8}{$X$}}
\psfrag{s}[Bc][Bc]{\scalebox{.8}{$\sigma_X$}}
\varsigma_X \, = \, \rsdraw{.45}{.9}{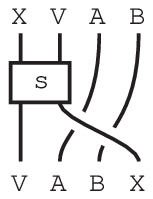} \,.
$$
For any morphism $f$ in $\zz(\bb)$, set  $d_A(f)=f \otimes \id_{A \otimes \ldual{A}}$. Then $d_A$ is clearly an endofunctor of $\zz(\bb)$.

\begin{figure}[t]
\begin{center}
\psfrag{V}[Bc][Bc]{\scalebox{.8}{$M$}}
\psfrag{A}[Bc][Bc]{\scalebox{.8}{$A$}}
\psfrag{B}[Br][Br]{\scalebox{.8}{$\ldual{A}$}}
\psfrag{C}[Bc][Bc]{\scalebox{.8}{$C$}}
\psfrag{s}[Bc][Bc]{\scalebox{.8}{$\sigma_A$}}
$\mu_{(M,\sigma)}$ \, = \, \rsdraw{.45}{.9}{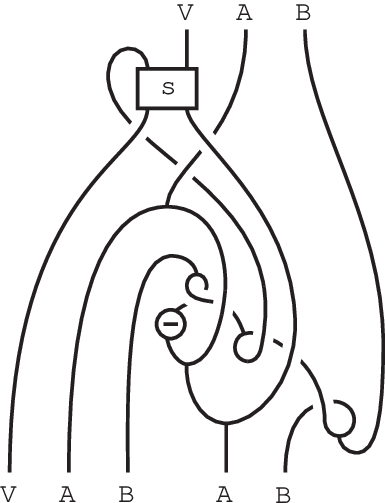}\;\,,
\qquad \,  $\eta_{(M,\sigma)}$ \, = \, \rsdraw{.45}{.9}{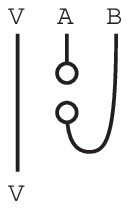}\;, \\[1em]
\psfrag{V}[Bc][Bc]{\scalebox{.8}{$M$}}
\psfrag{N}[Bc][Bc]{\scalebox{.8}{$N$}}
\psfrag{A}[Bc][Bc]{\scalebox{.8}{$A$}}
\psfrag{B}[Br][Br]{\scalebox{.8}{$\ldual{A}$}}
\psfrag{C}[Bc][Bc]{\scalebox{.8}{$C$}}
\psfrag{s}[Bc][Bc]{\scalebox{.8}{$\sigma_A$}}
$(d_A)_2\bigl((M,\sigma),(N,\omega)\bigr)$ \, = \, \rsdraw{.45}{.9}{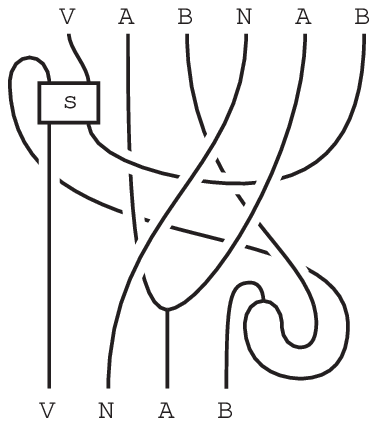},
\qquad \,  $(d_A)_0$ \, = \, \rsdraw{.45}{.9}{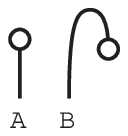}\;, \\[1em]
\psfrag{V}[Bc][Bc]{\scalebox{.8}{$M$}}
\psfrag{N}[Bc][Bc]{\scalebox{.8}{$N$}}
\psfrag{A}[Bc][Bc]{\scalebox{.8}{$A$}}
\psfrag{B}[Br][Br]{\scalebox{.8}{$\ldual{A}$}}
\psfrag{C}[Bc][Bc]{\scalebox{.8}{$C$}}
\psfrag{s}[Bc][Bc]{\scalebox{.8}{$\sigma_A$}}
$R_{(M,\sigma),(N,\omega)}$ \, = \, \rsdraw{.45}{.9}{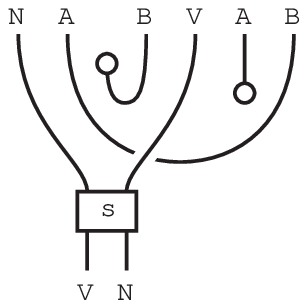} .
\end{center}
\caption{Structural morphisms of $d_A$}
\label{fig-mondA}
\end{figure}

\begin{thm}\label{thm-centraldouble}
The endofunctor $d_A$ is a quasitriangular Hopf monad on $\zz(\bb)$, with product $\mu$, unit $\eta$, comonoidal structure, and \Rt matrix $R$ given in Figure~\ref{fig-mondA}. Furthermore the functor
\begin{equation*}
\Psi\co\left \{
\begin{array}{ccc}
\zz(\bb)^{d_A} & \to & \zz(\bb_A) \\
\bigl((M,\sigma),\rho\bigr) & \mapsto & \bigl((M,r),\gamma\bigr)\\
f & \mapsto & f
\end{array}\right.
\end{equation*}
where
$$
\psfrag{V}[Bc][Bc]{\scalebox{.8}{$M$}}
\psfrag{N}[Bc][Bc]{\scalebox{.8}{$N$}}
\psfrag{A}[Bc][Bc]{\scalebox{.8}{$A$}}
\psfrag{B}[Br][Br]{\scalebox{.8}{$\ldual{A}$}}
\psfrag{s}[Bc][Bc]{\scalebox{.8}{$\sigma_A$}}
\psfrag{r}[Bc][Bc]{\scalebox{.9}{$\rho$}}
\psfrag{a}[Bc][Bc]{\scalebox{.9}{$s$}}
r \, = \, \rsdraw{.45}{.9}{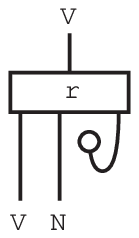} \qquad \text{and} \qquad \gamma_{(N,s)}\, = \, \rsdraw{.45}{.9}{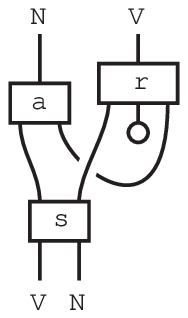} \;,
$$
is a braided strict monoidal isomorphism, with inverse given by
$$
\psfrag{V}[Bc][Bc]{\scalebox{.8}{$M$}}
\psfrag{N}[Bc][Bc]{\scalebox{.8}{$N$}}
\psfrag{A}[Bc][Bc]{\scalebox{.8}{$A$}}
\psfrag{B}[Br][Br]{\scalebox{.8}{$\ldual{A}$}}
\psfrag{s}[Bc][Bc]{\scalebox{1}{$\gamma_{(X, \id_X \otimes \varepsilon)}$}}
\psfrag{r}[Bc][Bc]{\scalebox{.9}{$\rho$}}
\psfrag{a}[Bc][Bc]{\scalebox{.9}{$s$}}
\sigma_X \, = \, \rsdraw{.45}{.9}{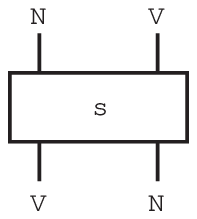} \quad \text{and} \quad
\psfrag{s}[Bc][Bc]{\scalebox{1}{$\gamma_{(A \otimes \ldual{A}, \alpha)}$}}
\psfrag{r}[Bc][Bc]{\scalebox{.9}{$r$}}
 \rho \, = \,  \rsdraw{.45}{.9}{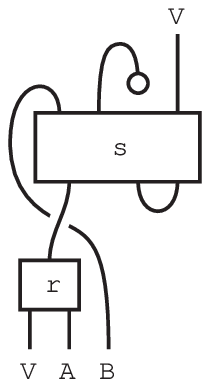}   \quad \text{where} \quad  \alpha \, = \,\rsdraw{.45}{.9}{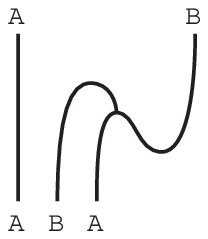}\,.
$$
Moreover, the
following triangle of monoidal functors
commutes:
$$\xymatrix@C=1pc{\zz(\bb)^{d_A}  \ar[rr]^\Psi\ar[rd]_{U_{d_A}} &\ar@{}[d]|(.4){\circlearrowright} & \zz(\bb_A) \ar[ld]^{\uu}\\ & \zz(\bb)}$$
\end{thm}
\begin{proof}
On verifies the theorem by direct computation. This is left to the reader.
\end{proof}

We call the quasitriangular Hopf monad $d_A$ the \emph{central double of $A$}.

\begin{rem}\label{rem-not-rep-dA}
From Corollary~\ref{prop-EX}, we see that the Hopf monad $d_A$ is not representable by a Hopf algebra in general. This can also be verified by hand, showing that, with the same category $\bb_n$ and algebra $A_n$ as in Section~\ref{sect-non-rep-res}, the Hopf monad $d_{A_n}$ admits no augmentation.
\end{rem}

\section{Cross products and cross quotients of Hopf monads}\label{sect-cp-cq}
In this section, we study the relationships between the double $D(A)$ and the central double $d_A$ of a Hopf algebra $A$. The tools used to this end are the cross product of Hopf monads (see \cite{BV3}) and the inverse operation, called the cross quotient (see \cite{BLV}). We refer to \cite{ML1} for detailed definitions of adjoint functors and adjunctions.

\subsection{Hopf monads and adjunctions}
The forgetful functor $U_T\co\cc^T \to \cc$ associated with a monad $T$ on a category $\cc$ has a left adjoint $F_T \co \cc \to \cc^T$, called the \emph{free module functor}, defined by $$F_T(X)=(T(X),\mu_X) \quad \text{and} \quad F_T(f)=T(f).$$
Conversely, let $(F \co \cc \to \dd, U \co \dd \to \cc)$ be an adjunction, with unit $\eta \co 1_\cc \to UF$ and counit $\varepsilon \co FU \to 1_\dd$. Then $T=UF$ is a monad with product $\mu=U(\varepsilon_F)$ and unit $\eta$.
There exists a unique functor $K \co \dd \to \mo{T}{\cc}$ such that $U_TK=U$ and $KF=F_T$. This functor $K$,  called
the \emph{comparison functor} of the adjunction $(F,U)$, is defined by $K(d)=(Ud, U\varepsilon_d)$.

An adjunction $(F,U)$ is \emph{monadic} if its comparison functor $K$ is an equivalence of categories.
For example, if $T$ is a monad on $\cc$, the adjunction $(F_T,U_T)$ has monad $T$ and comparison functor $K=1_{\mo{T}{\cc}}$, and so is monadic.

A \emph{comonoidal adjunction} is an adjunction
$(F\co \cc \to \dd, U\co \dd \to \cc)$, where $\cc$ and $\dd$ are monoidal categories, $F$ and $U$ are comonoidal functors, and the adjunction unit $\eta \co 1_\cc \to UF$ and counit $\varepsilon \co FU\to 1_\dd$ are comonoidal natural transformations.
Note  that if $(F,U)$ is a comonoidal adjunction, then $U$ is in fact a strong comonoidal functor, which we may view as a strong monoidal functor. Conversely, if a strong monoidal functor $U \co \dd \to \cc$ admits a left adjoint $F$, then $F$ is comonoidal,
with comonoidal structure given by
\begin{equation*}
F_2(X,Y)=\varepsilon_{FX \otimes FY} FU_2(FX,FY)F(\eta_X \otimes \eta_Y) \quad \text{and} \quad
F_0= \varepsilon_\un F(U_0),
\end{equation*}
and $(F,U)$ is a comonoidal adjunction (viewing $U$ as a strong comonoidal functor).

For example, the adjunction $(F_U,U_T)$ of a bimonad $T$ is a comonoidal adjunction (because $U_T$ is strong monoidal).
Conversely, the monad $T=UF$ of a comonoidal adjunction $(U,F)$ is a bimonad, and the comparison functor $K \co \dd \to \mo{T}{\cc}$ is strong monoidal and satisfies $U_T K=U$ as monoidal functors and  $KF=F_T$ as comonoidal functors (see  for instance  \cite[Theorem~2.6]{BV2}).

Let $(F\co \cc \to \dd, U\co \dd \to \cc)$ be a comonoidal adjunction between monoidal categories. The \emph{left Hopf operator} and the \emph{right Hopf operator} of $(F,U)$ are the natural transformations $\FO^l=\{\FO^l_{c,d}\}_{c\in \cc,d\in\dd}$ and $\FO^r=\{\FO^r_{d,c}\}_{d\in\dd,c\in \cc}$ defined by
\begin{align*}
&\FO^l_{c,d}=(\id_{F(c)}\otimes \varepsilon_d) F_2(c,U(d))\co F(c \otimes U(d)) \to F(c) \otimes d, \\
& \FO^r_{d,c} =(\varepsilon_d \otimes  \id_{F(c)} ) F_2(U(d),c)\co F(U(d) \otimes c) \to d \otimes F(c).
\end{align*}

A \emph{Hopf adjunction} is a  comonoidal adjunction whose left and right Hopf operators are invertible.
The monad of a Hopf adjunction is a Hopf monad. By \cite[Theore~~2.15]{BLV}, a bimonad is a Hopf monad if and only if its associated comonoidal adjunction is a Hopf adjunction.

The composite $(G,V)\circ (F,U)=(GF,UV)$ of two (comonoidal, Hopf) adjunctions is a (comonoidal, Hopf)  adjunction.

\subsection{Cross products}
Let $T$ be a monad on a category $\cc$. If $Q$ is a monad on the category $\mo{T}{\cc}$ of $T$\ti modules:
$$\mo{Q}{\bigl(\mo{T}{\cc}\bigr)}\adjunct{U_Q}{F_Q}\mo{T}{\cc}\adjunct{U_T}{F_T}\cc,$$
then the monad of the composite adjunction $(F_QF_T,U_TU_Q)$
is called the \emph{cross product of $T$ by $Q$} and denoted by $Q \cp T$ (see \cite[Section 3.7]{BV3}).  As an endofunctor of~$\cc$, $Q \cp T=U_TQF_T$. The product $p$ and unit $e$ of $Q \cp T$ are:
\begin{equation*}
p=q_{F_T} Q(\varepsilon_{Q F_T})  \quad \text{and}\quad e=v_{F_T}\eta,
\end{equation*}
where $q$ and $v$ are the product and the unit of $Q$, and $\eta$ and $\varepsilon$ are the unit and counit of the adjunction $(F_T,U_T)$. Note that the comparison functor of the composite adjunction $(F_QF_T,U_TU_Q)$  is a functor $$K \co (\mo{T}{\cc})^Q \to \mo{Q\cp T}{\cc}.$$
By \cite{BaWe}, if $Q$ preserves reflexive coequalizers, then $K$ is an isomorphism of categories.

%
%
%


If $T$ is a bimonad on a monoidal category $\cc$ and $Q$ is a bimonad on $\mo{T}{\cc}$, then  $Q\cp T=U_TQF_T$ is a bimonad on $\cc$ (since a composition of comonoidal adjunctions is a comonoidal adjunction), with comonoidal structure given by:
\begin{align*}
& (Q\cp  T)_2(X,Y)=Q_2\bigl(F_T(X),F_T(Y)\bigr)\, Q\bigl((F_T)_2(X,Y)\bigr),\\
& (Q\cp  T)_0=Q_0\, Q\bigl((F_T)_0\bigl).
\end{align*}
In that case the comparison functor $K \co (\mo{T}{\cc})^Q \to \mo{Q\cp T}{\cc}$
is strict monoidal.

By \cite[Proposition~4.4]{BLV}, the cross product of two Hopf monads is a Hopf monad.

\begin{exa}
Let $H$ be a bialgebra over a field $\kk$ and $A$ be a $H$-module algebra, that is, an algebra in the monoidal category $\lMod{H}$ of left $H$\ti modules. In this situation, we may form the cross product $A \rtimes H$, which is a $\kk$-algebra (see \cite{Maj2}). Recall that $H \otimes ?$ is a monad on $\Vect_\kk$ and $A \otimes ?$ is a monad on $\lMod{H}$. Then:
\begin{equation*}
(A \otimes ?)\cp (H \otimes ?)=(A\rtimes H) \otimes ?
\end{equation*}
as monads. Moreover, if $H$ is a quasitriangular Hopf algebra and $A$ is a $H$-module Hopf algebra, that is, a Hopf algebra in the braided category $\lMod{H}$, then $A\rtimes H$ is a Hopf algebra over $\kk$, and
$(A \otimes ?)\cp (H \otimes ?)=(A\rtimes H) \otimes ?$ as Hopf monads.
\end{exa}

\subsection{Cross quotients}\label{sect-croco}
Let $f \co T \to P$ be a morphism of monads on a category $\cc$. We say that $f$ is \emph{cross quotientable} if the functor $f^* \co \mo{P}{\cc} \to \mo{T}{\cc}$ is monadic. In that case, the monad of $f^*$ (on $\mo{T}{\cc}$) is called the \emph{cross quotient} of $f$ and is denoted by $P \cdiv_f T$ or simply $P \cdiv T$. Note that the comparison functor $K \co  \cc^P \to (\cc^T)^{P \cdiv T}$
is then an isomorphism of categories. 

By \cite[Remark~4.10]{BLV}, a morphism $f \co T \to P$ of monads on $\cc$  is cross quotientable whenever $\cc$ admits coequalizers of reflexive pairs and $P$ preserve them.

A cross quotient of bimonads is a bimonad: let $f \co T \to P$ be a cross quotientable morphism of bimonads on a monoidal category~$\cc$. Then $P \cdiv_f T$ is a bimonad on $\cc^T$ and the comparison functor
$K \co \mo{P}{\cc} \to (\mo{T}{\cc})^{P \cdiv_f T}$ is an isomorphism of monoidal categories.

The cross quotient is inverse to the cross product in the following sense: let $T$ be a (bi)monad on a (monoidal) category $\cc$.
If $T \to P$ is a cross quotientable morphism of (bi)monads on $\cc$, then $$(P \cdiv T) \cp T\simeq P$$ as (bi)monads.
Also, if $Q$ be a (bi)monad on $\mo{T}{\cc}$ such that $(F_QF_T,U_TU_Q)$ is monadic, then the unit  of $Q$ defines a cross quotientable morphism of (bi)monads $T \to Q\cp T$ and $$(Q\cp T)\cdiv T \simeq Q$$ as (bi)monads.

If $\cc$ is a monoidal category admitting reflexive coequalizers, and whose monoidal product preserves reflexive coequalizers, and if $T$ and $P$ are two Hopf monads on $\cc$ which preserve reflexive coequalizers, then
any morphism of bimonads $T \to P$  is cross quotientable and  $P \cdiv T$ is a Hopf monad (see \cite[Proposition~4.13]{BLV}).



\begin{exa}
Let $f \co L \to H$ be a morphism of Hopf algebras over a field $\kk$, so that $H$ becomes a $L$\ti bimodule by setting $\ell \cdot h \cdot \ell'=f(\ell)hf(\ell')$.  The morphism $f$ induces a morphism of Hopf monads on
$\Vect_\kk$:
 $$f \otimes_\kk ? \co L \otimes_\kk ? \to H \otimes_\kk ?$$
  which is cross quotientable,  and $(H \otimes ?)\cdiv (L \otimes ?)$  is a \kt linear Hopf monad on the monoidal category $\lMod{L}$ given by $N\mapsto H \otimes_L N$. (Note that in general this cross quotient is not representable by a Hopf algebra in the center of the category of left $L$\ti modules).
  This construction defines an equivalence of categories
$$L \backslash \HopfAlg_\kk \to \HopfMon_\kk(\lMod{L}),$$
  where $L \backslash\HopfAlg_\kk$ is the category of Hopf \kt algebras under $L$ and $\HopfMon_\kk(\lMod{L})$ is the category of \kt linear Hopf monads on $\lMod{L}$.
\end{exa}



\subsection{Applications to the doubles of $A$}
Let $A$ be a Hopf algebra in a braided rigid category admitting a coend $C$.
Recall from Theorems~\ref{thm-double-DA} and \ref{thm-centraldouble} that the double $D(A)$ of $A$ is a quasitriangular Hopf algebra in $\bb$ and that the central double $d_A$ of $A$ is a quasitriangular Hopf monad on $\zz(\bb)$ such that
$$
\zz(\bb_A) \simeq \bb_{D(A)} \simeq \zz(\bb)^{d_A}.
$$
Recall the quasitriangular Hopf algebras $D(A)$ and $C$ may be viewed as quasitriangular Hopf monads $? \otimes D(A)$ and $? \otimes C$ on $\bb$.

\begin{thm}\label{thm-crossDA}
We have: $$? \otimes D(A)=d_A \rtimes (? \otimes C)$$ as quasitriangular Hopf monads. In particular $$d_A=D(A)\cdiv_f C,$$ where $f\co C \to D(A)$ is the Hopf algebra morphism defined by $f=\id_{A \otimes \ldual{A}} \otimes u$ with $u$ the unit of $C$.
\end{thm}

\begin{proof}
The quasitriangular Hopf monads $? \otimes C$, $? \otimes D(A)$,  and $d_A$ are respectively the monads of the forgetful functor $U\co\co\zz(\bb) \to \bb$, the forgetful functor $U'\co\zz(\bb_A) \to \bb$, and the functor $\uu \co \zz(\bb_A) \to \zz(\bb)$. Since $U' = U \circ \uu$, we obtain that $? \otimes D(A)$ is the cross-product of $d_A$ by $? \otimes C$. This can be restated by saying that $d_A$ is the cross-quotient of $? \otimes D(A)$ by $? \otimes C$ along $? \otimes f$, or in short, $d_A = D(A) \cdiv_f C$.
\end{proof}

Recall from Remark~\ref{rem-not-rep-dA}, that the Hopf monad $d_A$ is not in general representable by a Hopf algebra. In particular Theorem~\ref{thm-crossDA} gives an illustration of the fact the cross-quotient of two Hopf monads representable by Hopf algebras is not always representable by a Hopf algebra. This may be explained by the fact that $C$ is not in general a retract of $D(A)$. Indeed, by \cite[Corollary~5.12]{BLV}, the cross-quotient $P \cdiv T$ of two Hopf monads is representable by a Hopf algebra in the center of the category of $T$-modules if and only if $T$ is a retract of $P$.

\end{document}